\documentclass[11pt,a4paper]{article}
\usepackage{amsmath,amsthm,amsopn,amsfonts,graphicx,amssymb,dsfont,color}
\usepackage[T1]{fontenc}
\usepackage[utf8]{inputenc}
\usepackage{lmodern}
\usepackage{mathrsfs}
\usepackage[dvipsnames]{xcolor}
\usepackage{enumitem}
\usepackage{color}
\usepackage{float}

\usepackage[margin=3cm]{geometry}
\newcommand\chrom[2]{\phantom{}^{{#1}}_{#2}}

\newtheorem{theo}{\textbf{Theorem}}[section]

\newtheorem{thm}[theo]{\textbf{Theorem}}
\newtheorem{lemma}[theo]{\textbf{Lemma}}
\newtheorem{lem}[theo]{\textbf{Lemma}}
\newtheorem{proposition}[theo]{\textbf{Proposition}}

\theoremstyle{definition}

\theoremstyle{remark}
\newtheorem{remark}[theo]{\textbf{Remark}}

\newcommand{\ep}{\varepsilon}
\newcommand{\R}{\mathbb R}
\newcommand{\rg}{{\rm rg}\,}
\newcommand\codim{{\rm codim}\,}
\newcommand\vspan{{\rm span}\,}
\newcommand\dd{\mathrm{d}}

\newcommand{\s}{\frac{\sigma^{2}}{2}}%{\frac{\sigma ^2}{2}}

\allowdisplaybreaks

\graphicspath{
	{figs/}
}

\title{The spatio-temporal dynamics of interacting genetic incompatibilities.\\
 Part {\rmfamily I}: The case of stacked underdominant clines
}
\date{}

\begin{document}
\maketitle

\begin{center}
	{\large\bf Matthieu Alfaro \footnote{Universit\'e de Rouen Normandie, CNRS, Laboratoire de Math\'ematiques Rapha\"el Salem, Saint-Etienne-du-Rouvray, France \& BioSP, INRAE, 84914, Avignon, France.}, Quentin Griette \footnote{Universit\'e de Bordeaux, CNRS, Institut de Math\'ematiques de Bordeaux, UMR 5251, 33400, Talence, France.}, Denis Roze \footnote{CNRS, IRL 3614 \& Sorbonne Universit\'{e}, Station Biologique de Roscoff, 29688 Roscoff, France.} and
Beno\^{i}t Sarels \footnote{Sorbonne Université, CNRS, Université de Paris, Laboratoire Jacques-Louis Lions (LJLL), 75005 Paris, France.}} \\
[2ex]
\end{center}

%\vspace{15pt}

%\tableofcontents

\vspace{5pt}

\begin{abstract}
We explore the interaction between two genetic incompatibilities (underdominant loci in diploid organisms) in a population occupying a one-dimensional space.
We derive a system of partial differential equations describing the dynamics of allele frequencies and linkage disequilibrium between the two loci, and use a quasi-linkage equilibrium approximation in order to reduce the number of variables.
We investigate the solutions of this system and demonstrate the existence of a solution in which the two clines in allele frequency remain stacked together. In the case of asymmetric incompatibilities ({\it i.e.} when one homozygote is favored over the other at each locus), these stacked clines propagate in  the form of a traveling wave. We obtain an approximation for the speed of this wave which, in particular, is  decreased by recombination between the two loci but is always larger than the speed of \lq\lq one cline alone''.

\medskip

\noindent \underline{Keywords:} genetic incompatibilities, underdominance, quasi-linkage equilibrium,  standing wave, traveling wave, perturbation analysis.\\

\noindent \underline{AMS Subject Classifications:} 92D10 (Genetics and epigenetics),  35C07 (Traveling wave solutions), 35B20 (Perturbations in context of PDEs).
\end{abstract}

\section{Introduction}\label{s:intro}

Genetic incompatibilities correspond to deleterious interactions between alleles (at the same locus or at different loci) within the same genome, and are the cause of the reduced fitness of hybrids between species \cite{CoyneOrr}, \cite{Gavrilets}. Such incompatibilities may be revealed by crosses between divergent populations or species, which may be performed experimentally \cite{Fraisse}, but may also occur naturally in hybrid zones resulting from secondary contacts between genetically divergent populations \cite{BartonHewitt1985, BartonHewitt1989}. Indeed, the offspring of such crosses carry a mix of alleles from the two parental populations, which may not function well together. Genetic incompatibilities may also be widespread within the same species, as suggested by recent data from \emph{Drosophila} \cite{CorbettDetig}. 

How several incompatibilities segregating within the same population interfere with each other has important consequences for the evolution of reproductive isolation, and the maintenance of distinct genetic entities after a secondary contact. Barton \& de Cara \cite{BartonDeCara} showed that, in the case of a single population (sympatry), incompatibilities are expected to couple through the buildup of linkage disequilibrium among them, which may eventually lead to strong reproductive isolation between two distinct genetic backgrounds. In spatially structured populations with limited dispersal (parapatry), clines in allele frequencies may form between regions containing different sets of incompatible alleles \cite{Barton79}. In this case, linkage disequilibria generated by dispersal generally tend to pull clines together \cite{Slatkin}, again leading to the {\it coupling} of genetic incompatibilities which then tend to reinforce each other. This process was explored by Barton \cite{Barton83} in the case of a continuous, linear habitat, and when genetic incompatibilities are generated by an arbitrary number of underdominant loci: at each locus, heterozygotes (say $Aa$) have a lower fitness than homozygotes, while the two homozygotes ($AA$ and $aa$) have equal fitness. This form of {\it symmetric} selection (against heterozygotes) can maintain stable clines in allele frequencies \cite{Bazykin}, \cite{Barton79}, separating regions where $AA$ and $aa$ individuals are abundant ($Aa$ hybrids being generated between these regions). When two such clines overlap in space (one separating regions where $AA$ and $aa$ are abundant, and the other separating regions where $BB$ and $bb$ are abundant), they tend to attract each other until they coincide, as illustrated by the numerical simulations in Figure \ref{fig:transitory-regime}, and then become motionless.

In the {\it asymmetric} situation where one allele has a selective advantage over the other (\emph{i.e.}, one homozygous genotype, say $AA$, has a higher fitness than the other, $aa$), the cline will move in the direction of the less fit genotype \cite{Barton79}. The interaction between several such asymmetric incompatibilities raises several questions that remain little explored to date, such as: when clines moving at different speeds come into contact, will they remain {\it stacked} (increasing the degree of reproductive isolation between the two sides of the clines) or not? If they do remain stacked, what will be the speed of the resulting front?  Under which conditions may an asymmetric incompatibility escape from a hybrid zone in which several incompatibilities are segregating?

This article constitutes a first step in the exploration of the spatio-temporal dynamics of interacting asymmetric incompatibilities, focusing on a simple situation involving two coupled underdominant loci (with alleles $A$, $a$ at the first locus, $B$, $b$ at the second) with identical fitness effects. Notice that Barton \cite{Barton83} has considered the situation where heterozygotes present a cost in fitness and where
the fitness of the homozygotes $AB|AB$ and the one of $ab|ab$ are the same, that is the symmetric case. In Section \ref{s:stationary-sol}, we give a mathematical proof of the existence of  such a cline in this symmetric situation, see Proposition \ref{prop:standing}.
We also prove in Proposition \ref{prop:standing-stab} that such a stationary cline is stable, {\it i.e.} that small perturbations of the profile may shift its spatial position but essentially do not alter its shape. In other words, a perturbed stationary cline comes back to a possibly shifted stationary cline. We refer to  \eqref{eq-ep-zero} for the he equation satisfied by such a cline.

When the fitness of the homozygote $AB|AB$ becomes slightly larger than the one of $ab|ab$, it is not {\it a priori} obvious whether the stationary cline stays stationary or begins to move. Here we answer this question by showing that invasion does occur even if the difference in fitness between homozygotes has a lower order of magnitude (measured by $0<\varepsilon\ll  1$)  compared to the fitness cost of heterozygotes. By using a rather involved perturbation analysis, we show in Theorem \ref{thm:existence-eps} that a front traveling at a constant speed $c_\varepsilon>0$ emerges from the stationary cline $u_0$ solving \eqref{eq-ep-zero} when $\varepsilon$ becomes positive. Such a traveling front is a solution to the reaction-diffusion equation involving nonlinear gradient terms  \eqref{eq-u}. We give an explicit approximation of the speed $c_\varepsilon$ which is, from the modelling point of view,  the main contribution of the present work. Among other implications, it reveals not only that recombination between the two loci tends to slow down the propagation of the front but also that the stacked clines always 
travel faster than one cline alone.

\medskip

The organization of the paper is as follows. In Section \ref{s:derivation} we derive the mathematical model, a PDE system involving nonlinear gradient terms. Through a phase plane analysis, we construct stationary solutions in Section \ref{s:stationary-sol}. Then, in Section \ref{s:front}, we construct traveling fronts thanks to a perturbation argument. In Section \ref{s:speed}, a trick enables to derive an explicit approximation of the speed $c_\ep$ which sheds light on the original model. We conclude and present some perspectives in Section \ref{s:conclusion}.

\section{Derivation of the mathematical model}\label{s:derivation}

\subsection{Biological assumptions}

The population is homogeneous in a one-dimensional space, along which density is supposed uniform and large.  We focus on the variations of the genetic composition of the population.
We consider that the fitness of a (diploid)  individual is affected by two loci: a first locus with two alleles $A$ and $a$ and a second locus with two alleles $B$ and $b$. We assume that heterozygotes have the lowest fitness (underdominance), the fitness of the different genotypes at each locus being given  in Table \ref{Table1}.

\begin{table}[H]
	\centering
\begin{tabular}{|c|c|}
	\hline
	genotype & fitness \\
	\hline \hline
	$AA$	& 	$1+2s_A$ 	\\
	$Aa$	&	$1+s_A-S_A$	\\
	$aa$	&	$1$	 	\\
	\hline
\end{tabular}
\hfil
\begin{tabular}{|c|c|}
	\hline
	genotype & fitness \\
	\hline \hline
	$BB$	& 	$1+2s_B$ 	\\
	$Bb$	&	$1+s_B-S_B$	\\
	$bb$	&	$1$	 	\\
	\hline
\end{tabular}
	\caption{ Fitness of the different genotypes at each locus.}\label{Table1}
\end{table}
\noindent Here the constants $s_A,s_B, S_A, S_B$ satisfy $0\leq s_A<S_A$, $0\leq s_B<S_B$. We then assume multiplicative effects among loci, so that the fitnesses {$\mathcal W$} of two-locus genotypes are given in Table \ref{Table2}.

\begin{table}[H]
\centering
	\begin{tabular}{|c||c|c|c|}
	\hline
	 & $AA$ & $Aa$ & $aa$ \\
	\hline \hline
	$BB$ & $ (1+2s_B)(1+2s_A)$ & $ (1+2s_B)(1+s_A-S_A)$ & $ 1+2s_B$ \\
	\hline
	$Bb$ & $ (1+s_B-S_B)(1+2s_A) $ & $(1+s_B-S_B)(1+s_A-S_A) $ & $ 1+s_B-S_B$ \\
	\hline
	$bb$ & $ 1+2s_A $ & $ 1+s_A-S_A$ & $1$ \\
	\hline
	\end{tabular}
	\caption{Fitness of diploid individuals.}\label{Table2}
\end{table}

\noindent Note that, because we will derive expressions to the first order in $s_A$, $s_B$, $S_A$ and $S_B$, assuming additive effects among loci would lead to the same results. 

We assume that recombination occurs with probability $r$, so that $AB|ab$ individuals may produce $Ab$ and $aB$ gametes.  Throughout the paper, the population occupying the left-hand side of the linear habitat will consist mostly of $AB|AB$ individuals, while the right-hand side will be mostly composed of $ab|ab$ individuals.
In the symmetric situation ($s_A=s_B=0$), if the clines of $A$ and $B$ are shifted in space, linkage disequilibrium will develop between the two loci and will pull both fronts together until they are stacked \cite{Slatkin}, \cite{Barton83}, as illustrated in Figure \ref{fig:transitory-regime}. This paper is concerned with the established regime where the fronts are stacked (this situation may also result from a secondary contact between two divergent populations, as considered by \cite{Barton83}). Note that in the general case ($s_A$, $s_B\neq 0$), shifted clines may not necessarily become stacked; however, we postpone the analysis of the conditions for stacking  to future works.

The PDEs describing the dynamics of two underdominant loci in a 1-dimensional continuous habitat can be obtained by combining the works  \cite{Barton79} and \cite{Barton83}. For the self-containedness of the present work, we present here a derivation of these equations, obtained by approximating a discrete-time model by a continuous-time model. In Section \ref{derivation} we present the genetic model that drives the genetic dynamics. In Section \ref{ss:spatial-model} we introduce the spatial structure and the corresponding equations. Finally in Section \ref{ss:conc-goals}, we precise our assumptions on the parameters and their respective magnitudes, as well as our precise objectives.

\subsection{Recursions on gamete frequencies in a discrete in time setting}
\label{derivation}

We start by considering a single population of diploid, hermaphroditic individuals with nonoverlapping generations. 
At the end of a generation (at time $t$), individuals release gametes and immediately die. The next generation, at time $t+1$, is formed by the random fusion of gametes. Under these hypotheses, it is sufficient to follow the frequencies of gametes produced at each generation, which completely determine the next generation of individuals (by the law of large numbers).
Denote by $y_{\chrom AB}$, $y_{\chrom Ab}$, $y_{\chrom aB}$, $y_{\chrom ab}$ the frequencies of the different types of gametes at generation $t$.  The fusion at random of these four combinations  gives birth to sixteen types of individuals (\lq\lq ordered'' in the sense that $z_{i|j}\neq z_{j|i}$ for $i\neq j$)
$$
z_{i|j}, \quad i, j\in\{\chrom AB, \chrom Ab, \chrom aB, \chrom ab\},
$$
with proportions $p_{i|j}$. Notice that, for $i\neq j$, the fusion can be male-female or female-male so that we  have $p_{i|j}=2\times \frac 12 y_iy_j$, thus
$$
p_{i|j}=y_{i}y_{j}.
$$

Each one of these individuals then produces gametes according to its fitness, providing the generation of gametes $y'_{\chrom AB}$, $y'_{\chrom Ab}$, $y'_{\chrom aB}$, $y'_{\chrom ab} $ at time $t+1$. Here we assume that there is a probability of recombination $0\leq r\leq \frac 12$ between the two loci. For each of the sixteen diploid genotypes, the process is as one of the three following examples:

$\bullet$ the individuals $z_{\chrom AB|\chrom AB}$, whose proportion is $y_{\chrom AB}^2$, release gametes $\chrom AB$ in proportion $1$.

$\bullet$ the individuals $z_{\chrom AB|\chrom Ab}$, whose proportion is $y_{\chrom AB}y_{\chrom Ab}$, release gametes $\chrom AB$ in proportion $\frac 12$ and gametes $\chrom Ab$ in proportion $\frac 12$.

$\bullet$ the individuals $z_{\chrom AB|\chrom ab}$, whose proportion is $y_{\chrom AB}y_{\chrom ab}$, release gametes $\chrom AB$ and $\chrom ab$ both in proportion $\frac{1-r}{2}$ (no recombination), and gametes $\chrom Ab$ and $\chrom aB$ both in proportion $\frac{r}{2}$ (recombination).

All these processes are weighted by the fitness of each type of individual, as in the above table. After a tedious but straightforward analysis, one obtains:
\begin{align*}
	\notag y'_{\chrom AB} =\frac{1}{\overline{\mathcal W}}&\Big[(1+2s_A)(1+2s_B)y_{\chrom AB}^2 + (1+2s_A)(1+s_B-S_B)y_{\chrom AB} y_{\chrom Ab} + (1+s_A-S_A)(1+2s_B) y_{\chrom AB}y_{\chrom aB} \\
	 &\quad + (1-r) (1+s_A-S_A)(1+s_B-S_B)y_{\chrom AB}y_{\chrom ab}+r(1+s_A-S_A)(1+s_B-S_B)y_{\chrom aB}y_{\chrom Ab}\Big]\\
	\notag y'_{\chrom Ab} =\frac{1}{\overline{\mathcal W}}&\Big[(1+2s_A) y_{\chrom Ab}^2 + (1+2s_A)(1+s_B-S_B) y_{\chrom Ab}y_{\chrom AB} + (1+s_A-S_A) y_{\chrom Ab}y_{\chrom ab} \\
	&\quad + (1-r)(1+s_A-S_A)(1+s_B-S_B)y_{\chrom Ab}y_{\chrom aB} + r(1+s_A-S_A)(1+s_B-S_B) y_{\chrom AB} y_{\chrom ab} \Big]\\
	\notag y'_{\chrom aB} =\frac{1}{\overline{\mathcal W}}&\Big[ (1+2s_B)y_{\chrom aB}^2 + (1+s_B-S_B)y_{\chrom aB}y_{\chrom ab} + (1+s_A-S_A)(1+2s_B)y_{\chrom aB}y_{\chrom AB} \\
	&\quad (1-r)(1+s_A-S_A)(1+s_B-S_B)y_{\chrom aB} y_{\chrom Ab} + r(1+s_A-S_A)(1+s_B-S_B)y_{\chrom AB}y_{\chrom ab} \Big]   \\
	\notag y'_{\chrom ab} =\frac{1}{\overline{\mathcal W}} &\Big[ y_{\chrom ab}^2 + (1+s_A-S_A)y_{\chrom ab}y_{\chrom Ab} + (1+s_B-S_B)y_{\chrom ab}y_{\chrom aB} \\
	&\quad + (1-r)(1+s_A-S_A)(1+s_B-S_B)y_{\chrom ab} y_{\chrom AB} + r(1+s_A-S_A)(1+s_B-S_B)y_{\chrom aB}y_{\chrom Ab}\Big],
	\end{align*}
where $\overline{\mathcal W}$ is the average fitness:
\begin{eqnarray*}
\overline{\mathcal W}&=&\sum_{i, j\in\{\chrom AB,\chrom Ab, \chrom aB, \chrom aB\}} z_{i|j}\mathcal W_{i|j}\\
&=& (1+2s_A)(1+2s_B)y_{\chrom AB}^2
+(1+2s_A)y_{\chrom Ab}^2
+(1+2s_B)y_{\chrom aB}^2
+y_{\chrom a b}^2 \\
&+& 2(1+2s_A)(1+s_B-S_B)y_{\chrom AB}y_{\chrom Ab}
+2(1+s_A-S_A)(1+2s_B)y_{\chrom AB}y_{\chrom aB} \\
&+&2(1+s_A-S_A)(1+s_B-S_B)y_{\chrom AB}y_{\chrom ab}
+2(1+s_A-S_A)(1+s_B-S_B)y_{\chrom Ab}y_{\chrom aB} \\
&+&2(1+s_A-S_A)y_{\chrom Ab}y_{\chrom ab}
+2(1+s_B-S_B)y_{\chrom aB}y_{\chrom ab}.
\end{eqnarray*}
Notice that adding the four above equations, one can check $
y'_{\chrom AB}+y'_{\chrom Ab} + y'_{\chrom aB} + y'_{\chrom ab} =\frac{\overline {\mathcal W}}{\overline {\mathcal W}}=1$.

\medskip

For ease of notation, we now let
\begin{equation}
\label{notations}
u:=y_{\chrom AB},\quad  v:=y_{\chrom Ab}, \quad w:=y_{\chrom aB}, \quad z:=y_{\chrom ab},
\end{equation}
so that
\begin{equation}
\label{somme-un}
u+v+w+z=1.
\end{equation}
As in \cite{Barton83}, we shall rather work on the three components system satisfied by
\begin{equation}
\label{inconnues}
p:=u+v, \quad q:=u+w, \quad  D:=uz-vw,
\end{equation}
where

$\bullet$  $p$ measures the frequency of allele $A$, 

$\bullet$ $q$ measures the frequency of allele $B$,

$\bullet$ $D$ stands for the linkage disequilibrium, measuring the association between alleles $A$ and $B$ within gametes (notice than, equivalently, $D=u-pq$).

Notice that 
\begin{equation}
\label{inverse}
u=pq+D,\quad v=p\left(1-q\right)-D,\quad w=\left(1-p\right)q-D, \quad z=(1-p)(1-q)+D.
\end{equation}

Next, we assume that $s_A$, $s_B$, $S_A$, $S_B$, $r$ are small and of the same order of magnitude, that is
\begin{equation}
\label{small-des}
s_A\leftarrow s_A\alpha,\quad s_B\leftarrow s_B \alpha,\quad S_A\leftarrow S_A\alpha,\quad S_B\leftarrow S_B\alpha,\quad r\leftarrow r\alpha, 
\end{equation}
for $0<\alpha\ll 1$. Taking into account \eqref{notations}, \eqref{somme-un},  \eqref{inconnues}, \eqref{inverse}, \eqref{small-des}, one can perform  straightforward (but tedious)  computations and obtain to the first order in $\alpha$: 
\begin{equation}\label{syst-pqD-discret}
    \begin{cases}
    p' &= p+\alpha\Big[\left(S_A(2p-1)+s_A\right)p(1-p)+\left(S_B(2q-1)+s_B\right)D\Big] \vspace{7pt}\\
   q' &= q+\alpha\Big[\left(S_B(2q-1)+s_B\right)q(1-q)+\left(S_A(2p-1)+s_A\right)D\Big]  \vspace{7pt}\\
   D' &=D-\alpha\Big[r+(2p-1)\left(S_A(2p-1)+s_A\right)+(2q-1)\left(S_B(2q-1)+s_B\right)\Big]D.
    \end{cases}
\end{equation}

\subsection{Inserting a spatial structure and switching to continuous time}
\label{ss:spatial-model}

Finally we consider the associated problem with a spatial structure $x\in \R$ (corresponding to the position of individuals along space) and continuous time $t\geq 0$. More precisely, we assume that gametes migrate according to a dispersal kernel centered on 0 and with variance $\sigma^{2}$. In the diffusion limit, and from \eqref{syst-pqD-discret}, the equations for the frequencies $p=p(t,x)$, $q=q(t,x)$ are
\begin{equation*}
    \begin{cases}
    p_t &= \s p_{xx} +\left(S_A(2p-1)+s_A\right)p(1-p)+\left(S_B(2q-1)+s_B\right)D \vspace{7pt}\\
   q_t &= \s q_{xx}+\left(S_B(2q-1)+s_B\right)q(1-q)+\left(S_A(2p-1)+s_A\right)D,
    \end{cases}
\end{equation*}
where $\sigma >0$. Notice that, since we assumed that the density of individuals is uniform and large, no advection term appears in the above system.  As for the equation for the disequilibrium  $D=uz-vw$, we  have additional gradient terms (e.g. \cite{BartonGale}, \cite{Barton83}) since
\begin{eqnarray*}
D_t &=&\left(\s u_{xx}+\cdots\right)z+u\left(\s z_{xx}+\cdots\right)-\left(\s v_{xx}+\cdots\right)w-v\left(\s w_{xx}+\cdots\right)\\
&=& \s \left(D_{xx}+2(-u_xz_x+v_xw_x)\right)+\cdots\\
&=&\s \left(D_{xx}+2(p_xq_x\right)+\cdots
\end{eqnarray*}
where we have used the identity
$$
p_xq_x=(u+v)_x(u+w)_x=u_x(u_x+v_x+w_x)+v_xw_x=-u_xz_x+v_xw_x.
$$
Hence, from \eqref{syst-pqD-discret}, the equation for $D=D(t,x)$ is
$$
D_t =\s D_{xx}+\sigma ^2 p_xq_x-\left[r+(2p-1)\left(S_A(2p-1)+s_A\right)+(2q-1)\left(S_B(2q-1)+s_B\right)\right]D.
$$

\subsection{Conclusion and goals}
\label{ss:conc-goals}

Hence the system for the allele frequencies $p=p(t,x)$, $q=q(t,x)$ and the linkage disequilibrium $D=D(t,x)$ is written
\begin{equation}\label{syst-p-q-D}
    \begin{cases}
    p_t &= \s p_{xx} +\left(S_A(2p-1)+s_A\right)p(1-p)+ (S_B(2q-1)+s_B) D \vspace{7pt}\\
   q_t &= \s q_{xx}+\left(S_B(2q-1)+s_B\right)q(1-q)+(S_A(2p-1)+s_A) D \vspace{7pt}\\
   D_t &=\s D_{xx}+\sigma ^2 p_xq_x-\left[r+(2p-1)\left(S_A(2p-1)+s_A\right)+(2q-1)\left(S_B(2q-1)+s_B\right)\right]D,
    \end{cases}
\end{equation}
where $\sigma >0$, $r>0$, $s_A>0$, $s_B>0$, $s_A>0$ and $S_B>0$ are given parameters. Observe that, starting from $D\equiv 0$ (no disequilibrium) the dynamics of $p$ and $q$ are decoupled but the gradient terms $p_x$ and $q_x$ in the $D$-equation cause disequilibrium and thus interaction \cite{Barton83}, see Figure \ref{fig:transitory-regime}.

\begin{figure}[H]
\begin{center}
\includegraphics[width=7cm]{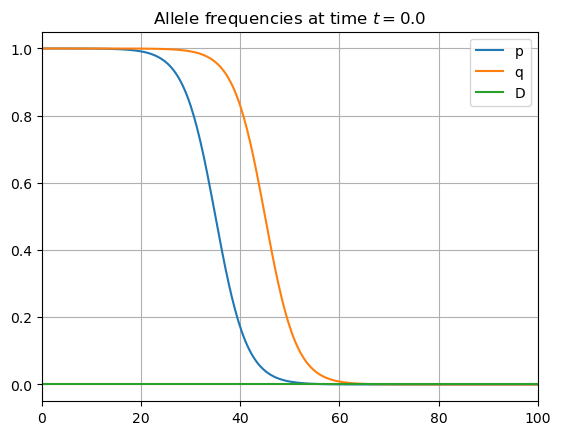}
\includegraphics[width=7cm]{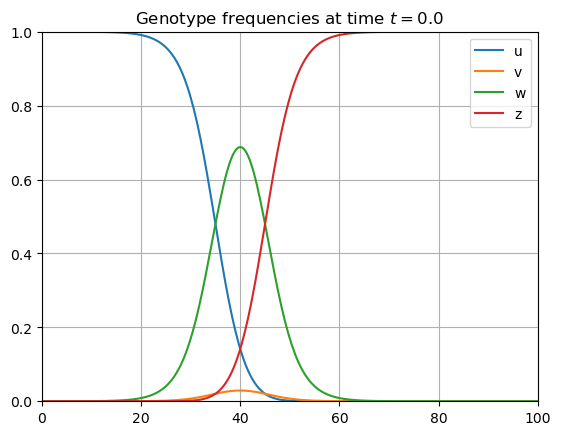}
\\
\includegraphics[width=7cm]{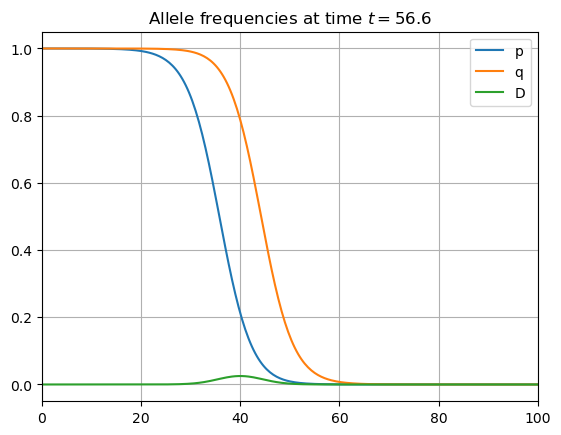}
\includegraphics[width=7cm]{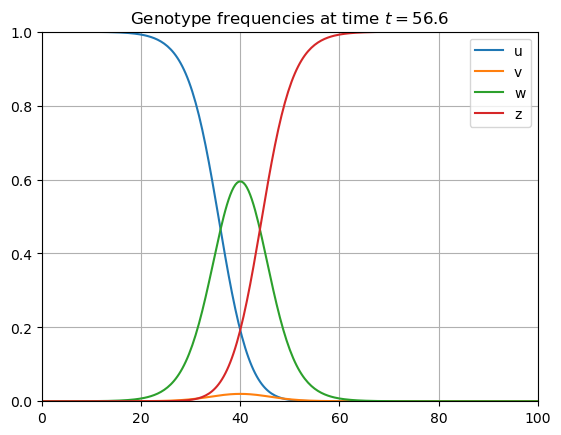}
\\
\includegraphics[width=7cm]{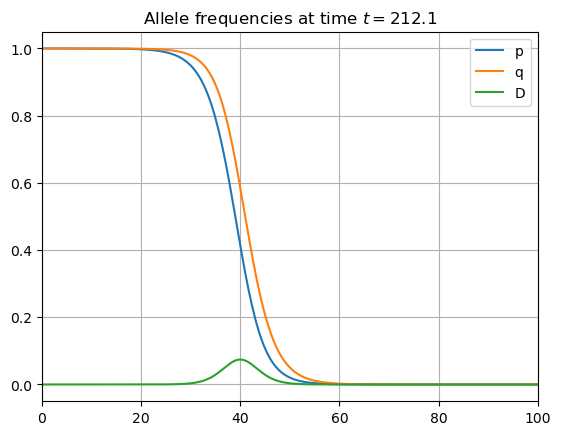}
\includegraphics[width=7cm]{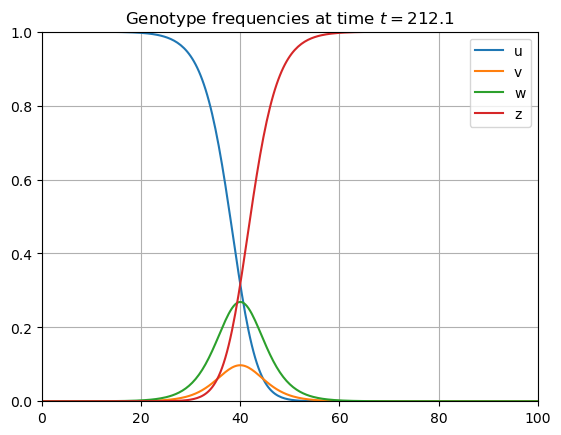}
\caption{ Numerical solutions with parameters $s_A=s_B=0$ (symmetric case), $S_A=S_B=r=0.1$ and $\s=1$. Left column:  $p$, $q$ and $D$; the clines are initially uncoupled; next, in a transitory regime, they are driven closer to each other, and eventually become stacked. Right column: the original unknowns, that is the frequencies of gametes $u$, $v$, $w$, $z$. Remark: the partial differential system in $(u, v, w, z)$ is a reaction-diffusion system for which a standard Strang splitting
method was used; numerical simulations were done in python 3.9.2 with the
NumPy package version 1.20.1.}
\label{fig:transitory-regime}
\end{center}
\end{figure}

Assuming that recombination $r$ is sufficiently large relative to the strength of selection against heterozygotes ($S_A$, $S_B$, determining the gradients in allele frequencies, e.g. \cite{Barton79}), one expects  that $D$ approximately follows
\begin{equation*}
	D_t \approx \s D_{xx}+ \sigma ^2 p_xq_x-rD.
\end{equation*}
In the sequel, we use a \emph{quasi-linkage equilibrium} approximation \cite{Barton83}, meaning that  the dynamics on $D$ is much faster than the one of $p$ and $q$. As a result,
$$
\s D_{xx}+ \sigma ^2 p_xq_x-rD\approx 0,
$$
whose solution is given by (one may use the Fourier transform to see it)
$$
D(t,x)\approx \frac{\sigma ^2}{r} \rho_\sigma * (p_xq_x(t,\cdot))(x), \quad \rho_\alpha(x):=\frac  12 \sqrt{\frac{2r}{\sigma ^2}}e^{-\sqrt{\frac{2r}{\sigma ^2}}\vert x\vert}.
$$
For $\sigma$ sufficiently small, the kernel $\rho _\alpha$ \lq\lq approaches'' the Dirac delta function, and thus
\begin{equation}
\label{quasi-ed}
D\approx \frac{\sigma ^2}{r}p_xq_x.
\end{equation}

As a result, using \eqref{quasi-ed} and writing $(p,q)(t,x)=(\tilde p,\tilde q)\left(t,\frac{\sqrt 2}{\sigma}x\right)$, we reach a simplified version of system \eqref{syst-p-q-D}, namely 
\begin{equation*}
    \begin{cases}
 \tilde   p_t &=  \tilde p_{xx} +S_A f( \tilde p)+ s_A g( \tilde p)+\frac 2 r (S_B(2 \tilde q-1)+s_B)  \tilde p_x  \tilde q_x,\vspace{7pt}\\
     \tilde q_t &=  \tilde q_{xx}+S_B f( \tilde q)+s_B g( \tilde q)+\frac 2 r(S_A(2 \tilde p-1)+s_A)  \tilde p_x  \tilde q_x, 
    \end{cases}
\end{equation*}
where
$$
f(u):=u(2u-1)(1-u), \quad g(u):=u(1-u).
$$
For ease of notation in the mathematical analysis, we now drop the tildes but keep in mind that, when returning to the original model, the traveling waves speeds we will find have to be multiplied by the factor $\frac{\sigma}{\sqrt 2}$. Last, we assume that
\begin{equation}\label{sAsBep}
S_A=S_B=S,\quad s_A=s_B=s=:\ep,
\end{equation}
and thus focus on the system
\begin{equation}\label{syst-p-q}
    \begin{cases}
    p_t &= p_{xx} +Sf(p)+ \ep g(p)+\frac 2 r (S(2q-1)+\ep) p_x q_x,\vspace{7pt}\\
   q_t &= q_{xx}+Sf(q)+\ep g(q)+\frac 2 r(S(2p-1)+\ep) p_x q_x.
    \end{cases}
\end{equation}
Notice that $f$ is a balanced bistable nonlinearity, which is slightly unbalanced by the term $\ep g$. 

In the sequel, our goal is to inquire on the situation where the $A$ cline, measured by $p$, and the $B$ cline, measured by $q$, remain stacked together.
To do so we look after $u=p=q$ solving the nonlinear equation
\begin{equation}
\label{eq-u}
u_t=u_{xx}+Sf(u)+\ep g(u)+\frac 2 r (S(2u-1)+\ep) u_x^2.
\end{equation}
We suspect the existence of a stationary solution connecting 1 to 0 for $\ep=0$ and that of a front connecting 1 to 0 and traveling at a speed $c_\ep\sim c_1 \ep$ for  some $c_1>0$ and $0<\ep \ll 1$. These facts are proved in Section \ref{s:stationary-sol} and \ref{s:front}, while $c_1$ is explicitly identified in Section \ref{s:speed}.

\section{Standing together ($\ep=0$)}\label{s:stationary-sol}

In this section, we construct a stationary solution connecting 1 to 0 in \eqref{eq-u} when $\ep =0$, and then prove its stability.

\subsection{Construction of the standing wave}

We are here looking after a $u_0:\R \to \R$ solving 
\begin{equation}
\label{eq-ep-zero}
\begin{cases}
u_0 ''+Sf(u_0)+\displaystyle \frac 2 r S(2u_0-1)(u_0')^2=0\quad  \text{ on } \R,\vspace{5pt}\\
u_0(-\infty)=1, \quad u_0(+\infty)=0.
\end{cases}
\end{equation}

\begin{lemma}[A priori estimates]\label{lem:a-priori} Any standing wave solution of \eqref{eq-ep-zero} has to satisfy $0<u_0<1$ and $u_0'(\pm \infty)=0$.
\end{lemma}

\begin{proof}  If $u_0 \leq 1$ is not true then, from the boundary conditions, $u_0$ has to reach a maximum value strictly larger than 1 at some point but, testing the equation at this point, this cannot hold.  Hence $u_0\leq 1$ and, from the strong maximum principle, $u_0<1$. Similarly $u_0>0$.

From the equation and the boundary condition, $u_0''>0$ in some $(A,+\infty)$, so that $u_0'$ is increasing on $(A,+\infty)$. As a result $u'_0$ has a limit in $+\infty$, which has to be zero since $u_0$ is bounded. Similarly $u_0'(-\infty)=0$.
\end{proof}

Using a phase plane analysis $(x,y)=(u_0,u_0')$, the equation in \eqref{eq-ep-zero} is recast
\begin{equation}
\label{phase-plane}
\begin{cases}
x'=y\vspace{5pt}\\
y'=-Sf(x)- \frac 2 rS(2x-1) y^{2}.
\end{cases}
\end{equation}
\begin{figure}
\begin{center}
  \includegraphics[width=7cm]{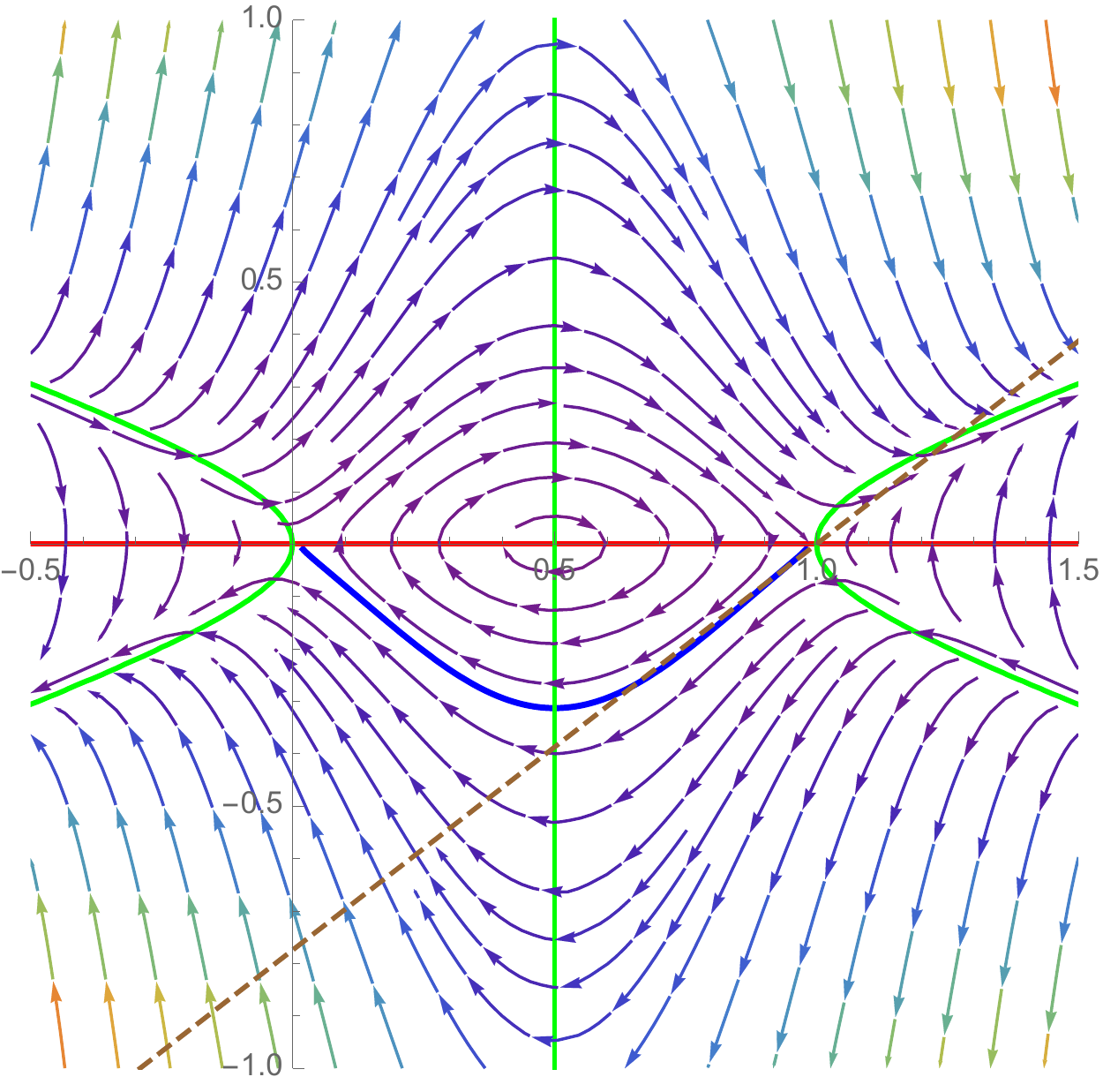}\quad \includegraphics[width=7cm]{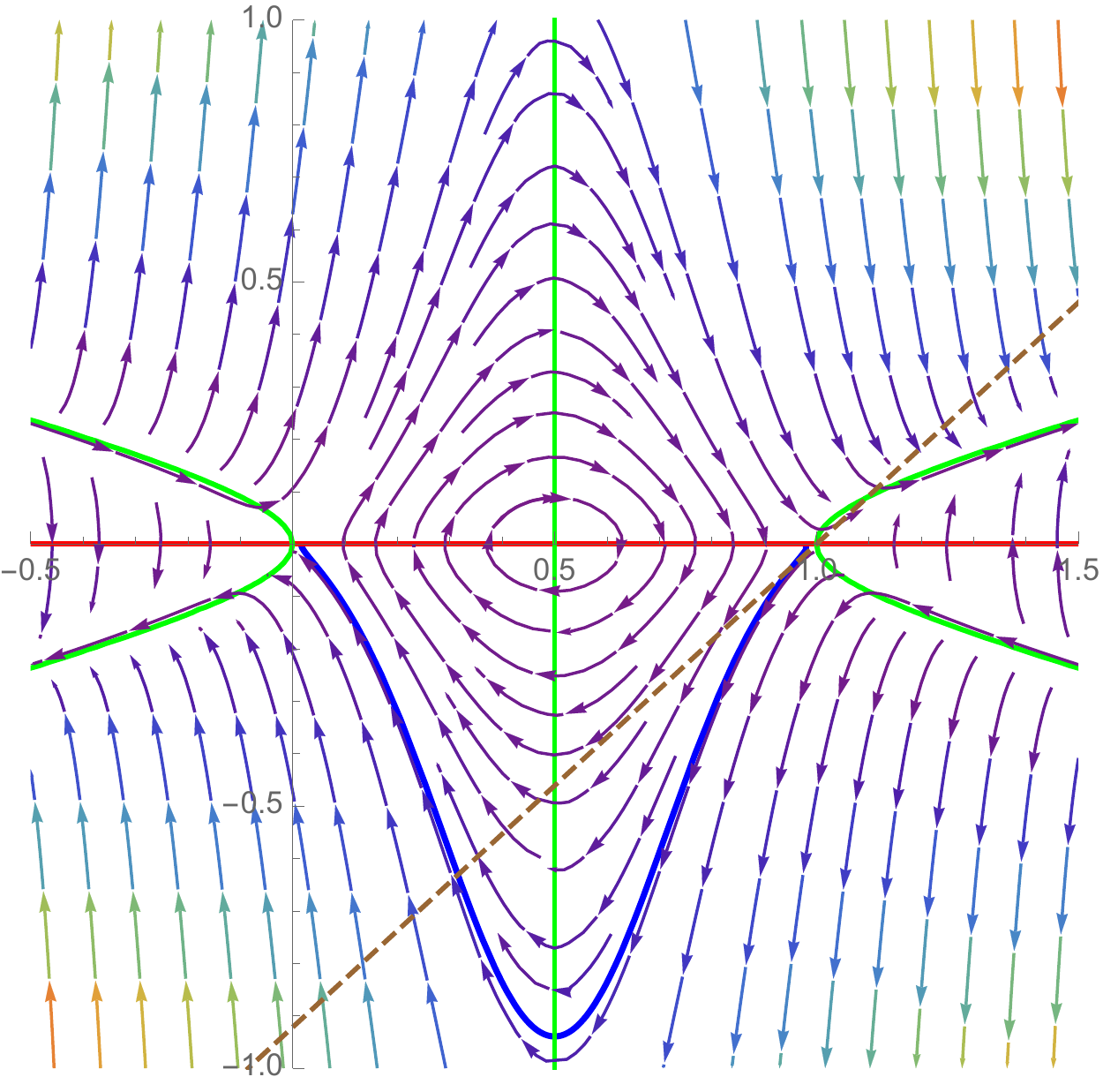}
    \caption{Phase plane analysis for \eqref{phase-plane}. In red, the nullcline $x'=0$, in green the nullcline $y'=0$, in brown dashed the linear unstable manifold at $(1,0)$, in blue (an approximation of) the heteroclinic orbit from $(1,0)$ to $(0,0)$. Left: the parameters are $S=0.6$, $r=0.25$ so that \eqref{condition} holds. Right: the parameters are $S=0.85$, $r=0.15$ so that \eqref{condition} does not hold.}
   \label{fig:phase-plane}
   \end{center}
\end{figure}
The phase plane analysis is depicted in  Figure \ref{fig:phase-plane}. The equilibria $(0,0)$ and $(1,0)$ are saddle points, the eigenvalues of the Jacobian matrix at these points being $\pm \sqrt S$, whereas the equilibrium $(\frac 1 2,0)$  is a center, the eigenvalues of the Jacobian matrix at this point being $\pm i\sqrt{\frac S2}$. At equilibrium $(1,0)$ the linear unstable manifold is the line $y=\sqrt{S}(x-1)$. To prove the existence of a heteroclinic orbit from $(1,0)$ to $(0,0)$, we consider the orbit leaving $(1,0)$ along the unstable manifold. As long as it has not reached $x=\frac 12$ this trajectory satisfies $x'<0$ and $y'<0$ (south west trajectory). In order to prove that the trajectory does cross the vertical line $x=\frac 12$, we need to construct a barrier, from below, preventing the situation $x\to l\geq \frac 12$, $y\to -\infty$. We choose the line $y=\alpha(x-1)$ with $\alpha>0$ to be selected large enough. Choosing $\alpha>\sqrt{S}$ insures that the trajectory is above the barrier in a neighborhood of $(1,0)$. We thus need to show that 
$$
\frac{\vert y'\vert}{\vert x'\vert}<\alpha \quad \text{ on the points $(x,y)$ such that  $y=\alpha(x-1)$, $\frac 12 \leq x <1$.}
$$
After some straightforward computations, this is recast
$$
\varphi(x):=
(2x-1)\left\vert\left(1-\frac{2\alpha ^2}{r}\right)x+\frac{2\alpha^2}{r}\right\vert<\frac{\alpha ^2}{S},\quad \text{ for all }  \frac 12 \leq x<1.
$$
Assuming $1-\frac{2\alpha ^2}{r}<0$, and evaluating the maximum of $\varphi$ on $[\frac 1 2,1]$, we reach
$$
\frac{\left(\frac{2\alpha ^2 }{r}+1\right)^2}{8\left(\frac{2\alpha ^2}{r}-1\right)}<\frac{\alpha ^2}{S},
$$
which can be obtained with $\alpha$ sufficiently large provided
\begin{equation}
\label{condition}
S<4r.
\end{equation}
Notice that, from the modelling point of view, assumption \eqref{condition} is consistent with the asymptotics \lq\lq $S$ small'' performed in Section \ref{s:derivation} (quasi-linkage equilibrium approximation). On the other hand, even if \eqref{condition} does not hold, the  (right) phase plane analysis of Figure \ref{fig:phase-plane} suggests that the heteroclinic orbit joining $(1,0)$ to $(0,0)$ still exists, but the above argument does not apply.

As a result, under assumption \eqref{condition}, the orbit touches the line $x=\frac 12$ at some point $(\frac 12,-\beta)$ for some $\beta>0$. Since the problem is symmetric with respect to $x=\frac 12$, we conclude that the orbit then converges to the equilibrium $(0,0)$ along the stable manifold, the linear stable manifold being given by $y=-\sqrt S x$. This trajectory provides a positive and decreasing solution $u_0$ to \eqref{eq-ep-zero}.

In other words, we have (nearly) proved the following.

\begin{proposition}[Stationary solution for $\ep=0$]\label{prop:standing}
Let us assume \eqref{condition}. Then there is  a unique $u_0:\R\to \R$ solving \eqref{eq-ep-zero} and satisfying the normalization condition $u_0(0)=\frac 12$. 

Moreover, $u_0$ is positive, decreasing, symmetric in the sense that
$$
u_0(-x)=1-u_0(x)\quad \text{ for all } x\in \R,
$$
 and has the asymptotics
\begin{equation}
\label{asymptotics}
1-u_0(x)\sim Ce^{\sqrt S x}\text{ as } x\to -\infty,\quad  u_0(x)\sim C e^{-\sqrt S x}\text{ as } x\to +\infty,
\end{equation}
for some $C>0$.
\end{proposition}

\begin{proof} From the above phase plane analysis, we are already equipped with a positive, decreasing and symmetric $u_0$ solving \eqref{eq-ep-zero}. The asymptotics \eqref{asymptotics} is rather classical but, for the convenience of the reader, we sketch a short and direct proof. We work as $x\to +\infty$. We know from the phase plane analysis that $u_0'(x)\sim -\sqrt S u_0(x)$ so that 
\begin{equation}\label{asympt-nulle}
u_0(x)=e^{-\sqrt S x +o(x)}.
\end{equation} Now, from the nonlinear ODE, we have, for some $K>0$,
\begin{equation*}
-Ku_0^2(x)\leq u_0''(x)-Su_0(x)\leq K u_0^2(x).
\end{equation*}
Multiplying this by $u_0'(x)<0$ and integrating from $x$ to $+\infty$, we have, 
$$
-\frac K 3 u_0^3(x)\leq -\frac 12 (u_0')^2(x)+\frac S 2 u_0^2(x)\leq \frac K 3 u_0^3(x),
$$
so that, for some $M>0$,
\begin{equation}\label{bidule}
-Mu_0^2(x)\leq u_0'(x)+\sqrt S u_0(x)\leq Mu_0^2(x).
\end{equation}
From this and \eqref{asympt-nulle} we deduce that $e^{\sqrt{S}x}(u_0'(x)+\sqrt{S}u_0(x))=\frac{d}{dx}\left(e^{\sqrt{S}x}u_0(x)\right)$ must be integrable in $+\infty$. As a result there is $C\geq 0$ such that $e^{\sqrt{S}x}u_0(x)\to C$ as $x\to +\infty$. Now the left inequality in \eqref{bidule} implies 
$$
	-\sqrt S \leq \frac{u_0'}{u_0+\frac{M}{\sqrt S}u_0^{2}}=\frac{u_0'}{u_0}-\frac{\frac{M}{\sqrt S}u_0'}{1+\frac{M}{\sqrt S}u_0}.
$$
Integrating this from $0$ to $x$ provides $\frac{u_0(0)}{1+\frac{M}{\sqrt{S}}u_0(0)}$ as a positive lower bound for $e^{\sqrt{S}x}u_0(x)$ so that $C>0$ and we are done with \eqref{asymptotics}.

It remains to prove uniqueness. We use a sliding method argument. Let $v_0$ be \lq\lq another'' solution such that $v_0(0)=\frac 12$.  For $K\geq 0$, define the shifted function  $v_K(x):=v_0(x-K)$.  Since $v_0$ must also have some asymptotics of the form \eqref{asymptotics}, say with some constant $C'>0$ instead of $C$, we see that $u_0\leq v_K$ on $\R$ for $K>0$ sufficiently large. As a result the real number
$$
K_0:=\inf \left\{K\in \R: u_0(x)\leq v_K(x), \forall x\in \R\right\}
$$
is well defined and nonnegative. Assume by contradiction that $K_0>0$. Then there is a point $x_0\in \R $ where $u_0(x_0)=v_{K_0}(x_0)$ and $u_0'(x_0)=v_{K_0}'(x_0)$ so that, from Cauchy-Lipschitz theorem, $u_0\equiv v_{K_0}$ on $\R$, which is excluded by the normalization conditions. As a result $K_0=0$ and thus $u_0\leq v_0$. Similarly $v_0\leq u_0$ and we are done.
\end{proof}

\subsection{Stability of the standing wave}

We prove here that the standing wave constructed in Proposition \ref{prop:standing} is linearly stable in the $L^\infty$ norm. More precisely the following holds.

\begin{proposition}[Stability of standing waves]\label{prop:standing-stab} Let $u_0$ be the standing wave  constructed in Proposition \ref{prop:standing}. Let $h\in C^1_b(\mathbb R)$ be given. Let $v$ solve the parabolic Cauchy problem
\begin{equation*}
    \begin{cases}
v_t(t, x)=v_{xx}(t, x)+Sf(v(t,x))+\displaystyle \frac 2 r S(2v(t,x)-1)(v_x(t,x))^2,  & t>0, x\in\mathbb R,\\
			v(0, x)=u_0(x)+\varepsilon h(x),   & x\in \R. 
	\end{cases}
	\end{equation*}

Then there is $\lambda_0>0$ such that, for any $0<\lambda<\lambda_0$, the following holds:  for sufficiently small $\ep$, there is  a continuous function $\gamma(\varepsilon)$ satisfying
$$
\gamma(0)=\int_{\mathbb R}h(x)u_0'(x)e^{\frac{4S}{r}(u_0^2(x)-u_0(x))}\dd x,
$$
 and a constant $K>0$ such that, for all $t>0$,
	\begin{equation*}
		\Vert v(t, \cdot)-u_0(\cdot +\varepsilon\gamma(\varepsilon))\Vert_{C^1_b(\mathbb R)}\leq Ke^{-\lambda t}.
	\end{equation*}
\end{proposition}

\begin{proof} We aim at applying a result of Sattinger, namely \cite[Theorem 4.1]{Sat-76}. To do so, we need to show that the linear operator (obtained by linearizing \eqref{eq-ep-zero} around the solution $u_0$)
	\begin{equation*}
		Lh:=h''+\frac{4S}{r}(2u_0-1)u_0'h' +S\left(f'(u_0)+\frac{4}{r}(u_0')^2\right)h,
	\end{equation*}
	satisfies the assumptions $(i)$ and $ (ii)$ of \cite[Lemma 3.4]{Sat-76}. Since  equation \eqref{eq-ep-zero} is a scalar quasilinear second-order differential equation set on $\mathbb R$ and with a smooth nonlinearity, the assumption $(ii)$ of \cite[Lemma 3.4]{Sat-76} can be readily checked thanks to \cite[Lemma 5.4]{Sat-76}. As for the assumption $(i)$ of \cite[Lemma 3.4]{Sat-76}, we point out that \cite[Corollary 5.7]{Sat-76} does not apply to our situation, and we thus  need to determine the spectrum of  $L$. 

	The liner operator $L$ admits $u_0'$ as principal eigenvector with eigenvalue 0. 	We remark that $L$ can be written as
	\begin{equation*}
		Lh=e^{-\frac{2S}{r}(u_0^2-u_0)}M\left(he^{\frac{2S}{r}(u_0^2-u_0)}\right),
	\end{equation*}
	where 
	\begin{equation*}
		Mk:=k''+\left(\frac{2S^2}{r}(2u_0-1)f(u_0)+Sf'(u_0)\right)k=:k''+c(x)k.
	\end{equation*}
	Since the weight function $e^{\frac{4S}{r}(u_0^2-u_0)}$ is bounded and uniformly positive, the operators $L$ and $M$ can be considered as acting on the same space $C^0_b(\mathbb R)$. In particular, $\lambda I-L$ admits a bounded inverse if and only if $\lambda I-M$ does (where $I$ is the identity mapping on $C^0_b(\mathbb R)$), and we have 
	\begin{equation*}
		(\lambda I-L)^{-1}=e^{-\frac{2S}{r}(u_0^2-u_0)}(\lambda I-M)^{-1}e^{\frac{2S}{r}(u_0^2-u_0)}.
	\end{equation*}
Below, by following ideas of \cite{Sat-76},  we analyze, for $g\in C^0_b(\mathbb R)$, the set of solutions to the resolvent equation 
	\begin{equation}\label{eq:resolvent}
		(\lambda I-M)k=-k''+(\lambda-c(x))k=g(x),
	\end{equation}
	 and then determine  the spectrum of $M$.\medskip
	
	\noindent{\bf 1. System of fundamental solutions to the homogeneous equation:} we first look for a system of fundamental solutions to
		\begin{equation}\label{eq:resolvent-bis}
		-k''+(\lambda-c(x))k=0,
	\end{equation}	
	whose behaviour near $\pm\infty$ can be determined (see \cite[Lemma 5.1]{Sat-76} for related arguments) for $\lambda\in \mathbb C$ such that $\lambda+S\not\in\mathbb R^-$.
	
	 Near $+\infty$, this is performed by substituting $\varphi_1(x)=z_1(x) e^{-\gamma_+ x}$ in \eqref{eq:resolvent-bis}, where $\gamma_+\in\mathbb C$ solves $\gamma_+^2=\lambda+S$ and $\mathrm{Re}\,\gamma_+>0$. We obtain
\begin{equation}
		-z_1''+2\gamma_+ z_1'-(S+c(x))z_1=0, \label{eq:ode-z_1}
		\end{equation}
		which is recast
		$$ 
		-(z_1'e^{-2\gamma_+x})'-(S+c(x))z_1e^{-2\gamma_+x}=0,
		$$
		so that, assuming $z_1'(+\infty)=0$,
		\begin{equation}
		z_1'(x)=\int_x^{+\infty}e^{-2\gamma_+(y-x)}(S+c(y))z_1(y)\dd y, \label{eq:z_1'}
		\end{equation}
		and thus, assuming $z_1(+\infty)=1$,
		\begin{equation}
		z_1(x)=1+\int_x^{+\infty}\frac{e^{2\gamma_+(x-y)}-1}{2\gamma_+}(S+c(y))z_1(y)\dd y.\label{eq:fundamental-fp}
	\end{equation}
Hence $z_1$ is written as the solution of a fixed-point problem \eqref{eq:fundamental-fp} set on $C^0_b(\mathbb R^+)$. Notice that the asymptotic behaviour \eqref{asymptotics}
of $u_0$ implies $y\mapsto S+c(y)\in L^1(\mathbb R^+)$. As a result, for a given $x_0>0$, the  right-hand side operator appearing in \eqref{eq:fundamental-fp} is globally Lipschitz continuous on $C^0_b([x_0, +\infty))$ with Lipschitz constant $\frac{1}{2|\gamma_+|}\int_{x_0}^{+\infty}|S+c(y)|\dd y$. Hence, equation \eqref{eq:fundamental-fp} has a unique solution $z_1$ on $C^0_b([x_0, +\infty)) $ for $x_0$ sufficiently large, and  this $z_1$ can be extended to $(-\infty,x_0)$ by solving the adequate Cauchy problem associated with \eqref{eq:ode-z_1}. We have therefore  constructed a solution $\varphi_1(x)=z_1(x)e^{-\gamma_+ x}$ to \eqref{eq:resolvent-bis} with $z_1\in C^0_b(\mathbb R^+)$, $z_1(+\infty)=1$.

 By the same procedure, but integrating on $[x_0, x]$ instead of $[x, +\infty)$ in \eqref{eq:z_1'}, we can construct a solution $\varphi_2(x)=z_2(x)e^{\gamma_+ x}$ to \eqref{eq:resolvent-bis} with $z_2\in C^0_b(\mathbb R^+)$ provided by the fixed-point problem
	\begin{equation*}
		z_2(x)=1+\int_{x_0}^x\frac{1-e^{-2\gamma_+(x-y)}}{2\gamma_+}(S+c(y))z_2(y)\dd y.
	\end{equation*}
	By the continuous dependence of the fixed-point with respect to the parameter $x_0$ \cite[Proposition 1.2]{Zei-86}, and by selecting $x_0$ sufficiently large, $z_2(x)$ can be made arbitrarily close to $1$. Indeed $z_2(x+x_0)$  is the unique fixed point of the operator 
	\begin{equation*}
		T_{x_0}z(x):=1+\int_{0}^x\frac{1-e^{-2\gamma_+(x-y)}}{2\gamma_+}(S+c(x_0+y))z(y)\dd y,
	\end{equation*}
	and $T_{x_0}$ converges uniformly to the constant operator $T_{+\infty}z\equiv 1$ as $x_0\to+\infty$:
	\begin{equation*}
		\Vert T_{x_0}z -1\Vert_{C^0_b([0, +\infty))}\leq\left(\frac{1}{2|\gamma_+|} \int_{x_0}^{+\infty}|S+c(y)|\dd y\right)\Vert z\Vert_{C^0_b([0, +\infty))}\xrightarrow[x_0\to\infty]{}0.
	\end{equation*}

	Therefore we have found a system of fundamental solutions $(\varphi_1, \varphi_2)$ to \eqref{eq:resolvent-bis} whose behaviour near $+\infty $ is known. We can proceed similarly near $-\infty$ and find another system of fundamental  solutions $(\psi_1,\psi_2)$ whose behaviour near $-\infty$ is known.

	Summarizing, for each $\lambda \in \mathbb C\setminus (-\infty,-S] $, we have 
	\begin{align}
		\varphi_1(x)&\approx_{+\infty}e^{-\gamma_+ x}, & \varphi_2(x)&\approx_{+\infty}e^{\gamma_+x},&
		\psi_1(x)&\approx_{-\infty}e^{\gamma_+ x}, & \psi_2(x)&\approx_{-\infty}e^{-\gamma_+x},\label{qqch-bidule} \\
		\varphi_1'(x)&\approx_{+\infty}e^{-\gamma_+ x}, & \varphi_2'(x)&\approx_{+\infty}e^{\gamma_+x},&
		\psi_1'(x)&\approx_{-\infty}e^{\gamma_+ x}, & \psi_2'(x)&\approx_{-\infty}e^{-\gamma_+x}, \label{qqch-bidule-bis}
	\end{align}
	where 	$A(x)\approx_{+\infty}B(x)$ means $0<\liminf_{x\to +\infty} \frac{\vert A(x)\vert}{B(x)}\leq \limsup_{x\to +\infty} \frac{\vert A(x)\vert}{B(x)}<+\infty$. 
Notice that, if $\lambda $ is not an eigenvalue of $M$, we further know that $\varphi_1$  is unbounded as $x\to -\infty$  (or else it would be an eigenvector), and $\psi_1$ is unbounded as $x\to +\infty$. Notice also that the constants involved in the above estimates are locally uniform in $\lambda$.
	
	\medskip

	\noindent{\bf 2. Solving  equation \eqref{eq:resolvent} if $\lambda\in \mathbb C\setminus (-\infty,-S] $ is not an eigenvalue of $M$}:  from the behaviours near $-\infty$,  the functions $\varphi_1$ and $\psi_1$ are linearly independent. Therefore, up to redefining $\varphi_2= \psi_1$, we may consider that $(\varphi_1, \varphi_2)$ is a system of fundamental solutions satisfying
	\begin{align*} 	
		\varphi_1(x)&\approx_{+\infty}e^{-\gamma_+ x}, & \varphi_2(x)&\approx_{+\infty}e^{\gamma_+x},&
		\varphi_1(x)&\approx_{-\infty}e^{-\gamma_+ x}, & \varphi_2(x)&\approx_{-\infty}e^{\gamma_+x}, \\
		\varphi_1'(x)&\approx_{+\infty}e^{-\gamma_+ x}, & \varphi_2'(x)&\approx_{+\infty}e^{\gamma_+x},&
		\varphi_1'(x)&\approx_{-\infty}e^{-\gamma_+ x}, & \varphi_2'(x)&\approx_{-\infty}e^{\gamma_+x}. 
	\end{align*}
		
We use the method of variation of constants to solve \eqref{eq:resolvent} and straightforwardly reach
		\begin{equation*}
			k(x)=\left(C_1-\frac{1}{W}\int_{-\infty}^x\varphi_2(y)g(y)\dd y\right)\varphi_1(x)+\left(C_2-\frac{1}{W}\int_x^{+\infty}\varphi_1(y)g(y)\dd y\right)\varphi_2(x),
	\end{equation*}
where $C_1$ and $C_2$ are arbitrary constants and $W$ is the constant Wronskian  $W=W(x)=\varphi_1(x)\varphi_2'(x)-\varphi_1'(x)\varphi _2(x)$.
Therefore, there is a unique bounded solution $k(x)$, which corresponds to $C_1=C_2=0$. 

Hence, for each $g\in C^0_b(\R)$ there exists a unique $k\in C^2_b(\mathbb R)$ such that $(\lambda I- M)k=g$. By the open mapping theorem, the operator $\lambda I-M$ has a bounded inverse $(\lambda I-M)^{-1}:C^0_b(\mathbb R)\to C^2_b(\mathbb R)\hookrightarrow C^0_b(\mathbb R)$.
	In particular, 
	\begin{equation*}
\text{if }\lambda\in \mathbb C\setminus (-\infty,-S] \text{ is not an eigenvalue of } M \text{, then }\lambda \text{ is in the resolvent set of }M.
	\end{equation*}

\medskip

	\noindent{\bf 3. The eigenvalues in $\mathbb C\setminus (-\infty,-S]$ of $M$}: if $\lambda\in \mathbb C\setminus (-\infty,-S]$ is an eigenvalue of $M$ then, from \eqref{qqch-bidule}, the eigenvector must be proportional to both $\varphi_1$ and $\psi_1$, hence $\varphi_1$ and $\psi_1$ are not linearly independent. Hence the Wronskian $\varphi _1\psi_1'-\varphi_1'\psi_1$ must vanish. Since the Wronskian is analytic in $\lambda$ (see \cite[Lemma 5.2]{Sat-76}) and  not identically zero, the eigenvalues of $M$ in $\mathbb C\setminus (-\infty,-S]$ are isolated.
	
	Let $\lambda\in \mathbb C\setminus (-\infty,-S]$ be an eigenvalue of $M$. Then the associated eigenvector $\varphi$ is a solution to \eqref{eq:resolvent} and  the former analysis applies. In particular, $\varphi$ and $\varphi'$ converge exponentially fast to 0 near $\pm\infty$ (at rate $\mp\gamma_+$, $\mathrm{Re}\,\gamma_+>0$) and therefore $\varphi\in H^1(\mathbb R)$. Since $M$ is symmetric on $ H^1(\mathbb R)$, we have in fact $\lambda \in\mathbb R$. Reproducing the argument of \cite[Theorem 5.5]{Sat-76}, we see that there are no positive eigenvalues of $M$.

\medskip

	We conclude from the above analysis that the eigenvalues of $M$ in $\mathbb C\setminus (-\infty,-S]$ form a sequence $(\lambda_n)_{n\in \mathbb N}$ (with $\lambda_0=0$) of isolated values in $(-S, 0]$. As a result the spectrum of $M$ satisfies
	\begin{equation*}
		\sigma(L, C^0_b(\mathbb R))=\sigma(M, C^0_b(\mathbb R))\subset(-\infty,-S]\cup\{\lambda_n, n\geq 0\}.
	\end{equation*}
This shows that the assumption $(i)$ of \cite[Lemma 3.4]{Sat-76} holds in our case and concludes the proof of Proposition \ref{prop:standing-stab}.
\end{proof}

\section{Traveling together ($0<\ep\ll 1$)}\label{s:front}

In this section, we construct a traveling front connecting 1 to 0  in \eqref{eq-u}, when $0<\ep\ll 1$, through a perturbation argument from the case $\ep=0$ studied above.

\medskip

We are here looking after a nonnegative profile $u:\R \to \R$ and a speed $c\in \R$ solving 
\begin{equation}
\label{eq-ep}
\begin{cases}
u ''+cu'+Sf(u)+\ep g(u)+\frac 2 r (S(2u-1)+\ep)\displaystyle (u')^2=0\quad  \text{ on } \R,\vspace{5pt}\\
u(-\infty)=1, \quad u(+\infty)=0.
\end{cases}
\end{equation}
Observe that, from the strong maximum principle we have $u>0$. Also, as in the proof of Lemma \ref{lem:a-priori}, we have $u<1$. Hence, we {\it a priori} know $0<u<1$. 

We use a perturbation technique and look for $u$ in the form 
$$
u=u_0+h,
$$
where $u_0$ is provided by Proposition \ref{prop:standing} and with, typically, $h(\pm\infty)=h'(\pm \infty)=0$. Plugging this ansatz into the equation, we see that we need $\mathcal F(\ep,c,h)=0$, where $$\mathcal F:\R\times \R\times E \to \widetilde E$$ is defined by
\begin{align}
\mathcal F(\ep,c, h)&:=h''+cu_0'+ch'+S(f(u_0+h)-f(u_0))+\ep g(u_0+h)\notag\\
	&\phantom{:}\quad+\frac 2r\left(S(2u_0+2h-1)+\ep\right)(u_0'+h')^2 -\frac 2rS(2u_0-1)(u_0')^2. \label{eq:F}
\end{align}
As for the function spaces, we choose the weighted H\"{o}lder spaces 
\begin{equation}
\label{espaces}
E:=C^{2, \alpha}_\mu(\mathbb R), \quad \widetilde E:=C^{0, \alpha}_\mu(\mathbb R), \quad 0<\alpha <1,
\end{equation}
where, for $k\in \mathbb N$,
$$
	  C^{k, \alpha}_\mu(\R):=\left\{f\in C^{k}(\mathbb R):\Vert f\Vert_{C^{k, \alpha}_\mu(\R)}<+\infty\right\}, \quad \Vert f\Vert_{C^{k, \alpha}_\mu(\R)}:= \left\Vert x\mapsto e^{\mu\sqrt{1+x^2}}f(x)\right\Vert_{C^{k, \alpha}(\mathbb R)},
$$
for well-chosen $\mu\geq 0$. Here, $C^{k,\alpha}(\R)$ denotes  the H\"{o}lder space  consisting of functions of the class $C^k$, which are continuous and bounded on the real axis $\R$ together with their derivatives of order $k$, and such that the derivatives of order $k$ satisfy the H\"{o}lder condition with the exponent $0<\alpha<1$ . The norm in this space is the usual H\"{o}lder norm. 

Our main result in this section then reads as follows. 

\begin{thm}[Traveling waves for $0<\ep\ll 1$]\label{thm:existence-eps} Let $0\leq \mu <\sqrt S$ be given.  Let $\mathcal F:\R\times \R\times C^{2,\alpha}_\mu(\R)\to C^{0,\alpha}_\mu(\R)$ be defined as in \eqref{eq:F}. 

Then there is $\varepsilon_0>0$ such that, for any $0\leq \varepsilon\leq \varepsilon_0$, there exists $(c_\ep, h_\ep)\in \mathbb R\times E$ such that $\mathcal F(\varepsilon, c_\varepsilon, h_\varepsilon)=0$. Moreover the map $\varepsilon \mapsto(c_\varepsilon, h_\varepsilon)$ is continuous, the speed $c_\ep$ satisfies
	\begin{equation}\label{speed}
		c_\ep=\frac{\displaystyle -\int_\R \left(g(u_0)+\frac 2r (u_0')^{2}\right)u_0'e^{\frac{4S}{r}(u_0^2-u_0)}} { \displaystyle\int_\R (u_0')^{2}e^{\frac{4S}{r}(u_0^2-u_0)}}\,\varepsilon+o(\varepsilon), \quad \text{ as } \ep \to 0,
	\end{equation}
	whereas the perturbation profile $h_\ep$ satisfies
	\begin{equation}\label{perturbation}
	 \int_{\mathbb R}h_\varepsilon u_0' = 0, \quad \text{ for all } 0\leq \ep\leq \ep_0.
	 \end{equation}
\end{thm}

In what follows we aim at  applying the Implicit Function Theorem \ref{thm:implicit-functions} to the operator $\mathcal F$ defined in \eqref{eq:F}, see \cite{Apr-Duc-Vol-09} for a related argument. We straightforwardly compute the derivatives with respect to  $c$ and $h$ at the origin $(0,0,0)$:
$$
\partial_c\mathcal F(0, 0, 0)(c)=cu_0',
$$
and
\begin{equation}\label{eq:defL}
	Lh:=\partial_h\mathcal F(0, 0, 0)(h)=h''+\frac{4S}{r}u_0'(2u_0-1)h' + S\left(f'(u_0)+\frac{4}{r}(u_0')^2\right)h.
\end{equation}
We need to show that  $ \partial_{c,h}\mathcal F(0, 0, 0)$ given by
$$
	(c,h) \mapsto   Lh +cu_0'	
$$
 is bijective from and to a well-chosen pair of function spaces. Our strategy is as follows. In subsection \ref{ss:fred}, thanks to some results of \cite{Vol-Vol-Col-99}, \cite{Vol-11} (recalled in Appendix),  we show that $L$ is a Fredholm operator and compute its index (which depends on the choice of $\mu$). Next, in subsection \ref{ss:ker}, we determine the kernel of $L$. In particular $u_0'$ is the only bounded solution. We also determine the kernel of $L^*$ thanks to an algebraic symmetric formulation in a well-chosen weighted $L^2$ space, from which we deduce the surjectivity of $\partial _{c,h}\mathcal F(0,0,0)$. Then we conclude the proof of Theorem \ref{thm:existence-eps} in subsection \ref{ss:proof}. 

\subsection{Fredholm property}\label{ss:fred}

\begin{lem}[Fredholm property]\label{lem:fredprop}
	The operator $L:C^{2, \alpha}_\mu(\mathbb R)\to C^\alpha_\mu(\mathbb R)$, defined in \eqref{eq:defL}, is Fredholm  if $\mu\neq \sqrt{S}$ and  we have
	\begin{equation*}
		\text{\rm ind } L=\begin{cases} 0 & \text{ if }\; 0\leq\mu< \sqrt{S}, \\
		-2 & \text{ if }\; \mu >\sqrt{S}.
		\end{cases}
	\end{equation*}
\end{lem}
\begin{proof}
In view of  Remark \ref{rem:weighted-holder} it suffices to study the limiting operators $(L^\mu)^\pm$ associated with $L ^\mu$ defined as in \eqref{Lmu}, namely
$$
	(L^\mu)^\pm h=h''\mp 2\mu h' + (\mu^2-S)h,
$$	
thanks to Theorem \ref{thm:Fredholm-ODE}. First since $-\xi^{2}\mp 2\mu i\xi +\mu^{2}-S=0$, corresponding to \eqref{eq:test-Fredholm}, has no real solution, $L$ is Fredholm. Next, the associated characteristic equation, corresponding to  \eqref{eq:char}, writes
	\begin{equation*}
		X^2\pm 2\mu X+(\mu^2-S)=0,
	\end{equation*}	
		and has the following roots:
	\begin{align*}
		X^+_{1,2}&=-\mu \pm\sqrt{S},\\
		X^-_{1,2}&=+\mu \pm\sqrt{S}.
	\end{align*}
	If $0\leq \mu<\sqrt{S}$ we deduce that $\kappa^+=1$ and $\kappa^-=1$ (in the notations of Theorem \ref{thm:Fredholm-ODE}), hence $\text{\rm ind } L =0$; if $\sqrt{S}<\mu$ we have $\kappa^+=0$ and $\kappa^-=2$, hence $\text{\rm ind }L=-2$. This completes the proof of Lemma \ref{lem:fredprop}.
\end{proof}

%\begin{remark} Since we compute the kernels of $L$ and $L^*$ anyways, we could simply deduce the Fredholm index from the dimensions of the kernels by tail estimates. However, Lemma \ref{lem:fredprop} also shows that $L$ is Fredhom (which is not easy) and gives a way to check that our computations are correct.
%\end{remark}

\subsection{Kernels of $L$, $L^*$ and surjectivity of $\partial_{c,h}\mathcal F(0,0,0)$}
\label{ss:ker}

\begin{lem}[The kernel of $L$]\label{lem:kerL}
	Two linearly independent solutions to the linear homogeneous ordinary differential equation 
	\begin{equation}\label{eq:edo-kerL}
		Lh:=h''+\frac{4S}{r}u_0'(2u_0-1)h' + S\left(f'(u_0)+\frac{4}{r}(u_0')^2\right)h=0
	\end{equation}
 are given by  
	$$
		u_0' \quad\text{ and } \quad  v_0:x\mapsto u_0'(x)\int_0^x\frac{1}{(u_0')^2(z)}e^{-\frac{4S}{r}\left(u_0^2(z)-u_0(z)\right)}\dd z.
	$$
	Among the two, $u_0'$ is the only bounded solution.
	
	As a result, for $0\leq \mu<\sqrt S$,  the kernel of the operator $L$ acting on the space $C^{2, \alpha}_\mu(\mathbb R) $ into $C^{0, \alpha}_\mu(\R)$ is given by
	\begin{equation*}
		\ker L=\text{\rm span }u_0'.
	\end{equation*}
\end{lem}

\begin{proof}
	We investigate the solutions $h$ to \eqref{eq:edo-kerL}.
	This is a  second-order linear homogeneous ordinary differential equation, and we already know a solution $u_0'$ (as seen by differentiating \eqref{eq-ep-zero}).
		In this case a second solution $v_0$ can be sought in the form $v_0(x)=z(x) u_0'(x)$. Indeed plugging this ansatz into \eqref{eq:edo-kerL} yields the following first order linear ordinary differential equation for $z'$:
	$$
	z''+\left(2\frac{u_0''}{u_0'}+\frac{4S}{r}(2u_0-1)u_0'\right)z'=0,
	$$
or, equivalently, 		
$$
	z''+\left(\ln((u_0')^{2})+\frac{4S}{r}(u_0^{2}-u_0)\right)'z'=0.
$$	
As a result, we can select the solution 
$$
z'(x)=\frac{1}{(u_0')^2(x)}e^{-\frac{4S}{r}\left(u_0^2(x)-u_0(x)\right)},
$$
	which we integrate to reach $z(x)$, and thus
	\begin{equation}\label{def:v0-non-borne}
		v_0(x)=u_0'(x)\int_0^x\frac{1}{(u_0')^2(z)}e^{-\frac{4S}{r}\left(u_0^2(z)-u_0(z)\right)}\dd z.
	\end{equation}
Now, from the analysis in Section \ref{s:stationary-sol}, we know that, for some $C>0$,
\begin{equation}\label{uzeroprime-asymptotic}
u_0'(z)	\sim Ce^{-\sqrt S z}, \quad \text{ as } z\to +\infty.
\end{equation}
 Since $u_0(+\infty)=0$, the integrand in \eqref{def:v0-non-borne} is equivalent to $\frac{1}{C^2}e^{2\sqrt S z}$ as $z\to +\infty$, and thus
\begin{equation}\label{v0-asymptotic}
v_0(x)\sim \frac{1}{C2\sqrt S}e^{\sqrt S x}, \quad \text{ as } x\to +\infty.
\end{equation}
	Thus $v_0$ is unbounded and, in particular, $v_0\not\in C^{2, \alpha}_\mu (\mathbb R)$. Since solutions to \eqref{eq:edo-kerL} are the linear combinations of $u_0'\in C^{2,\alpha}_\mu(\R)$ when $0\leq \mu <\sqrt S$, $v_0\notin C^{2, \alpha}_\mu (\mathbb R)$, and since $L:C^{2, \alpha}_\mu (\mathbb R)\to C^{\alpha}_\mu (\mathbb R)$, we conclude that 
		$\ker L=\text{\rm span }u_0'$ when $0\leq \mu <\sqrt S$.
\end{proof}

\begin{lem}[The kernel of $L^*$]\label{lem:kerL*}
	If $0\leq \mu<\sqrt{S}$ then the kernel of the adjoint operator $L^*$ is 
	\begin{equation*}
		\ker L^*=\text{\rm span }\left(u_0'e^{\frac{4S}{r}(u_0^2-u_0)}\right).
	\end{equation*}
	On the other hand, if $ \mu>\sqrt{S}$ then 
	\begin{equation*}
		\ker L^*=\text{\rm span }\left(u_0'e^{\frac{4S}{r}(u_0^2-u_0)}, v_0e^{{\frac{4S}{r}(u_0^2-u_0)}}\right),
	\end{equation*}
	where $v_0(x):=u_0'(x)\displaystyle \int_0^x\frac{1}{(u_0')^2(z)}e^{-\frac{4S}{r}\left(u_0^2(z)-u_0(z)\right)}\dd z$ is as in Lemma \ref{lem:kerL}.
\end{lem}

\begin{proof}
	Our starting point is to notice that the coefficient of the first-order term in the definition of $L$, that is $u_0'(2u_0-1)$, is the derivative of $u_0^2-u_0$ so that 
	\begin{equation*}
		Lh=h''+\frac{4S}{r}(u_0^2-u_0)'h'+S\left(f'(u_0)+\frac{4}{r}(u_0')^2\right)h, 
	\end{equation*}
	from which we deduce the formulation
$$
		Lh=\left(h'e^{\frac{4S}{r}(u_0^2-u_0)}\right)'e^{-\frac{4S}{r}(u_0^2-u_0)}+ S\left(f'(u_0)+\frac{4}{r}(u_0')^2\right)h,
		$$
which is symmetric in the adequate weighted $L^2$ space:
\begin{align*}
			\int_{\mathbb R}k(Lh) e^{\frac{4S}{r}(u_0^2-u_0)}&=-\int_{\mathbb R}k'h'e^{\frac{4S}{r}(u_0^2-u_0)} +\int_{\mathbb R}S\left(f'(u_0)+\frac{4}{r}(u_0')^2\right)hke^{\frac{4S}{r}(u_0^2-u_0)}\\
			&=\int_{\mathbb R}(Lk)he^{\frac{4s}{r}(u_0^2-u_0)}.
\end{align*}
In particular, for any $k\in C^{2, \alpha}_\mu(\mathbb R)$, we have
	\begin{align*}
		\int_{\mathbb R}k(Lh)&=\int_{\mathbb R}k\left(h'e^{\frac{4S}{r}(u_0^2-u_0)}\right)'e^{-\frac{4S}{r}(u_0^2-u_0)}+ S\left(f'(u_0)+\frac{4}{r}(u_0')^2\right)hk\\
		&=\int_{\mathbb R} -\left(ke^{-\frac{4S}{r}(u_0^2-u_0)}\right)'h'e^{\frac{4S}{r}(u_0^2-u_0)} \\
		&\qquad +\int_{\mathbb R}S\left(f'(u_0)+\frac{4}{r}(u_0')^2\right)h\left(ke^{-\frac{4S}{r}(u_0^2-u_0)}\right)e^{\frac{4S}{r}(u_0^2-u_0)} \\
		&=\int_{\mathbb R}h\left(L(ke^{-\frac{4S}{r}(u_0^2-u_0)})\right)e^{\frac{4S}{r}(u_0^2-u_0)}.
	\end{align*}
	Therefore, if $ve^{-\frac{4S}{r}(u_0^2-u_0)}=k\in\ker L $, then we have 
	\begin{align*}
		\int_{\mathbb R}(L^*v) h&=\int_{\mathbb R}v (Lh) \\
		&=\int_{\mathbb R}h\left(L(ve^{-\frac{4S}{r}(u_0^2-u_0)})\right)e^{\frac{4S}{r}(u_0^2-u_0)}=0,
	\end{align*}
	provided each integral is finite. In particular, since $C^{2, \alpha}_\mu(\mathbb R)$ is dense in $C^{0, \alpha}_\mu(\mathbb R)$, this shows that
	$$
	\vspan\left(u_0'e^{\frac{4S}{r}(u_0^2-u_0)}\right)\subset \ker L^*.
	$$
	
	 Assume  $0\leq \mu< \sqrt{S}$. Then  we deduce from Lemma \ref{lem:fredprop} and Lemma \ref{lem:kerL} that $\dim\ker L^*=-\textrm{ ind } L+\dim\ker L =0+1=1$, and therefore we do have  $\ker L^*=\vspan\left(u_0'e^{\frac{4S}{r}(u_0^2-u_0)}\right)$.
	 
	  Assume $\mu>\sqrt{S}$. This time, the asymptotics for $v_0$ being given in \eqref{v0-asymptotic}, terms $\int_{\mathbb R}v_0he^{\frac{4S}{r}(u_0^2-u_0)}$ are finite as soon as $h\in C^{0, \alpha}_\mu(\mathbb R)$, and therefore
	  $$
	  \vspan\left(v_0e^{\frac{4S}{r}(u_0^2-u_0)}\right)\subset \ker L^*,
	  $$
	  by a density argument. Then  we deduce from Lemma \ref{lem:fredprop} and Lemma \ref{lem:kerL} that  $\dim \ker L^*=-\textrm{ ind } L+\dim\ker L=-(-2)+0=2$. Since $u_0'e^{\frac{4S}{r}(u_0^2-u_0)}$ and $v_0e^{\frac{4S}{r}(u_0^2-u_0)}$ are linearly independent, we do have $\ker L^*=\vspan\left(u_0'e^{\frac{4S}{r}(u_0^2-u_0)},v_0e^{\frac{4S}{r}(u_0^2-u_0)}\right)$. 
\end{proof}

\begin{lem}[Surjectivity of $\partial_{c,h}\mathcal F(0,0,0)$]\label{lem:onto}
	Let $0\leq \mu< \sqrt{S}$ be given. Then, the application 
\begin{eqnarray*}
\partial_{c, h}\mathcal F(0, 0, 0):& \R\times C_\mu ^{2,\alpha}(\R)  & \to \quad\;    C_\mu^{0,\alpha}(\R)\\
& (c,h)  &\mapsto\quad\,    Lh+cu_0'
\end{eqnarray*}
	is surjective.
\end{lem}

\begin{proof}
	We check that $u_0'$ is not in the range of $L$. Since $L$ has closed range we have $\rg{L}=(\ker L^*)^\perp$, and thus $\rg{L}=\left(\vspan\left(u_0'e^{\frac{4S}{4}(u_0^2-u_0)}\right)\right)^\perp$ from Lemma \ref{lem:kerL*}. But
	\begin{equation*}
		\left\langle u_0'e^{\frac{4S}{4}(u_0^2-u_0)}, u_0'\right\rangle_{(C^{0, \alpha}_\mu(\R))^*, C^{0, \alpha}_\mu(\R)}=\int_\R (u_0')^2e^{\frac{4S}{4}(u_0^2-u_0)}>0
	\end{equation*}
so that $u_0'\not\in \rg{L}$. Since $\rg L$ has codimension 1 by Lemma \ref{lem:fredprop} and \ref{lem:kerL}, we have $C^{0, \alpha}_\mu(\mathbb R)=\rg{L}\oplus \vspan u_0'$. This shows that $\partial_{c, h}\mathcal F(0, 0, 0) $ is surjective.
\end{proof}

\begin{remark} We present here an alternate  way to prove that $u_0'\notin \rg L$ remains true when $\mu\geq \sqrt{S}$. To do so, let us solve the second-order linear ordinary differential equation 
	\begin{equation}\label{eq:edo-second-membre}
		w''+\frac{4S}{r}u_0'(2u_0-1)w' + S\left(f'(u_0)+\frac{4}{r}(u_0')^2\right)w=u_0'.
	\end{equation}
	Recall that the solutions of the associated homogeneous equation are spanned by $u_0'$ and $v_0$ provided by Lemma \ref{lem:kerL}. To find a particular solution to \eqref{eq:edo-second-membre}, we use the method of variation of constants. We see that
	$\varphi(x):=\lambda _1 (x)u_0'(x)+\lambda _2(x)v_0(x)$ solves \eqref{eq:edo-second-membre} as soon as
	$$
	\begin{cases}
		u_0' \lambda _1'+v_0\lambda _2'&=0\\
		u_0'' \lambda _1'+v_0'\lambda _2'&=u_0',
	\end{cases}
	$$
	which yields
	$$
	\lambda_2'\frac{u_0'v_0'-u_0''v_0}{u_0'}=u_0', \quad \lambda _1'=-\frac{v_0}{u_0'}\lambda _2'.
	$$
	Since $u_0'v_0'-u_0''v_0$ is nothing else than the Wronskian, it is equal to $\theta^{-1} e^{-\frac{4S}{r}(u_0^{2}-u_0)}$ for some $\theta\neq 0$,  and thus
	$$
	\begin{cases}
		\lambda_2'(x)=\theta (u_0')^{2}(x)e^{\frac{4S}{r}(u_0^{2}(x)-u_0(x))}\sim \theta C^{2}e^{-2\sqrt S x}\vspace{5pt}\\
		\lambda_1'(x)=-\theta v_0(x)u_0'(x)e^{\frac{4S}{r}(u_0^{2}(x)-u_0(x))}\sim -\frac{\theta}{2\sqrt S},
	\end{cases}
	$$
	where the equivalents are taken as $x\to +\infty$ and where we have used \eqref{uzeroprime-asymptotic} and \eqref{v0-asymptotic}. Hence, we can select
	$$
	\begin{cases}
		\lambda_2(x)=-\int_x^{+\infty}\theta (u_0')^{2}(z)e^{\frac{4S}{r}(u_0^{2}(z)-u_0(z))}\dd z\sim \frac{\theta C^{2}}{2\sqrt S}e^{-2\sqrt S x}\vspace{5pt}\\
		\lambda_1(x)=-\int _0 ^{x}\theta v_0(z)u_0'(z)e^{\frac{4S}{r}(u_0^{2}(z)-u_0(z))}\dd z \sim -\frac{\theta}{2\sqrt S}x.
	\end{cases}
	$$
	Hence the solutions to \eqref{eq:edo-second-membre} are 
	$$
	w(x)=(C_1+\lambda _1(x))u_0'(x)+(C_2+\lambda_2(x))v_0(x)
	$$
	for any $C_1\in \R$, $C_2\in \R$. If $C_2\neq 0$ then, from all the above asymptotic, $w$ is unbounded. If $C_2=0$ then, from all the above asymptotics,
	$$
	w(x)\sim -\frac{\theta C}{2\sqrt S}xe^{-\sqrt{S}x}, \quad \text{ as } x\to +\infty.
	$$
This above asymptotics shows that $w\notin C^{2,\alpha}_\mu(\R)$ when $\mu \geq \sqrt S$, and  thus	$u_0'\notin \rg L$.	
\end{remark}

\subsection{Construction of traveling waves}\label{ss:proof}

We are now in the position to complete the proof of Theorem \ref{thm:existence-eps}, that is the construction of traveling waves for \eqref{eq-ep} when $0<\ep\ll 1$.

\begin{proof}[Proof of Theorem \ref{thm:existence-eps}] Assume $0\leq \mu<\sqrt S$. Let us recall that $\mathcal F:\mathbb R \times \mathbb R\times C^{2, \alpha}_\mu(\R)\to C^{0,\alpha}_\mu(\R)$ is  given by \eqref{eq:F}. It is Fréchet differentiable (even of the class $C^1$) with respect to each of its variables, and we have
	\begin{eqnarray*}
		\partial_\varepsilon\mathcal F (0, 0, 0) &= & g(u_0)+\frac{2}{r}(u_0')^2, \\
		\partial_c\mathcal F(0, 0, 0) &= & u_0' \\
		L=\partial_h\mathcal F(0, 0, 0) &: & h\mapsto Lh=h''+\frac{4S}{r}u_0'(2u_0-1)h' + S\left(f'(u_0)+\frac{4}{r}(u_0')^2\right)h.
	\end{eqnarray*}
	We have shown, in Lemma \ref{lem:fredprop}, that $L$ is a Fredholm operator with indice $0$ and, in Lemma \ref{lem:kerL}, that  the kernel of $L$ is $\vspan u_0'$ in the considered weighted Hölder space.

	Our concern is the derivative $\partial_{c, h}\mathcal F(0, 0,0): (c, h)\mapsto Lh+cu_0'$. It has been shown in Lemma \ref{lem:onto} that it is surjective. It is not difficult to show that 
	\begin{equation*}
		\ker \partial_{c, h}\mathcal F(0, 0,0) =\{0\}\times\vspan u_0', 
	\end{equation*}
	and that the restriction of $\partial_{c, h}\mathcal F(0, 0,0) $ to $\R\times N$, where 
	\begin{equation*}
		N:=\left\{f\in C^{2, \alpha}_\mu(\mathbb R): \int_\R f u_0'= 0\right\}%=(\ker L)^{\perp},
	\end{equation*}
is a topological complement of $\ker L$, is injective and still surjective. Therefore we can apply the Implicit Function Theorem \ref{thm:implicit-functions} to the restriction of $\mathcal F$ to $\mathbb R\times\mathbb R\times N$. We deduce the existence of a branch $(c_\ep, h_\ep)$, $0\leq \ep\ll 1$, of solutions with $\ep\mapsto (c_\ep,h_\ep)$ continuous and $h_\ep$ satisfying \eqref{perturbation}.

It remains to prove \eqref{speed}. Since $\mathcal F$ is $C^1$ in all its variables we deduce from $\mathcal F (\ep,c_\ep,h_\ep)=0$ and  the chain rule that
$$
\partial_\varepsilon \mathcal F(\varepsilon, c_\ep, h_\ep)+\frac{dc_\ep}{d\ep} \partial_c\mathcal F(\varepsilon, c_\ep, h_\ep) + \partial_h \mathcal F(\varepsilon, c_\ep, h_\ep)\left(\frac{d h_\ep}{d\ep}\right) = 0,
$$
which we evaluate at $\ep=0$ to get
$$
g(u_0)+\frac{2}{r}(u_0')^{2}+\left.\frac{dc_\ep}{d\ep}\right \vert _{\ep=0} u_0'+L\left(\left.\frac{d h_\ep}{d\ep}\right\vert_{\ep=0}\right)=0.
$$
Since $\rg L=(\ker L^*)^{\perp}=\left(\vspan \left(u_0'e^{\frac{4S}{r}(u_0^2-u_0)}\right)\right)^{\perp}$, multiplying the above by $u_0'e^{\frac{4S}{r}(u_0^2-u_0)}$ and integrating over $\R$, we reach
$$
\left.\frac{dc_\ep}{d\ep}\right \vert _{\ep=0}=\frac{\displaystyle -\int_\R \left(g(u_0)+\frac 2r (u_0')^{2}\right)u_0'e^{\frac{4S}{r}(u_0^2-u_0)}} { \displaystyle\int_\R (u_0')^{2}e^{\frac{4S}{r}(u_0^2-u_0)}}>0,
$$
which yields \eqref{speed} and concludes the proof of Theorem \ref{thm:existence-eps}.
\end{proof}

\section{The speed of the traveling stacked clines}\label{s:speed}

In this section, we obtain an explicit form for $c_1=c_1(r,S)$ appearing in the asymptotic formula for the speed $c_\ep=c_1\ep+o(\ep)$ as given in  \eqref{speed}. This in turn provides valuable insights on the model for coupled underdominant clines.
\medskip

For convenience let us temporarily denote $u=u_0$ the standing wave solution constructed in Proposition \ref{prop:standing}. From \eqref{eq-ep-zero}, we see that $v:={u'}^2$ solves the linear first order ODE
$$
v'+\frac{4S}{r}(u^2-u)'v=-2Su'(2u-1)u(1-u),
$$
which is solved as
$$
{u'}^2(x)=e^{-\frac{4S}{r}(u^2-u)(x)}\left(C-2S\int _0 ^x (u^{2}-u)'(t) e^{\frac{4S}{r}(u^2-u)(t)}(u-u^{2})(t)dt\right),
$$
for some constant $C$. We integrate by parts and, up to changing the value of the constant $C$, reach
\begin{eqnarray*}
{u'}^2(x)&=&e^{-\frac{4S}{r}(u^2-u)(x)}\left(C-\frac r 2\left(e^{\frac{4S}{r}(u^{2}-u)(x)}(u-u^{2})(x)+\int_0^{x}e^{\frac{4S}{r}(u^{2}-u)(t)}(u^{2}-u)'(t)dt\right)\right)\\
&=& e^{-\frac{4S}{r}(u^2-u)(x)}\left(C-\frac r 2\left(e^{\frac{4S}{r}(u^{2}-u)(x)}(u-u^{2})(x)+\frac{r}{4S}e^{\frac{4S}{r}(u^{2}-u)(x)}\right)\right).
\end{eqnarray*}
Letting for instance $x\to +\infty$ enforces $C=\frac{r^{2}}{8S}$ and, returning to the notation $u_0$, we finally obtain
\begin{equation}\label{u-prime-u}
{u_0'}^{2}(x)=\frac{r^{2}}{8S}e^{\frac{4S}{r}(u_0-u_0^{2})(x)}-\frac{r}{2}(u_0-u_0^{2})(x)-\frac{r^{2}}{8S}.
\end{equation}
The fact that $u_0'$ can be expressed in terms of $u_0$, already observed in \cite{Barton83}, enables to obtain an explicit form $c_1=c_1(r,S)$ appearing in $c_\ep=c_1\ep+o(\ep)$ as given in  \eqref{speed}. Indeed, using \eqref{u-prime-u} and recalling that $u_0'<0$, we obtain
$$
c_1=\frac{\displaystyle -\int_\R \frac{r}{4S}u_0'(x)\left(1-e^{-\frac{4S}{r}(u_0-u_0^2)(x)}\right)dx} { -\displaystyle\int_\R u_0'(x)\left(\frac{r^{2}}{8S}e^{\frac{4S}{r}(u_0-u_0^{2})(x)}-\frac{r}{2}(u_0-u_0^{2})(x)-\frac{r^{2}}{8S}\right)^{\frac 12}e^{-\frac{4S}{r}(u_0-u_0^2)(x)}dx}.
$$
	Performing the change of variable $u=u_0(x)$ this is recast
	$$
c_1=\frac{\displaystyle \int_0^{1} \frac{r}{4S}\left(1-e^{-\frac{4S}{r}(u-u^2)}\right)du} { \displaystyle\int_0^{1} \left(\frac{r^{2}}{8S}e^{\frac{4S}{r}(u-u^{2})}-\frac{r}{2}(u-u^{2})-\frac{r^{2}}{8S}\right)^{\frac 12}e^{-\frac{4S}{r}(u-u^2)}du}. 
$$
Expanding with respect to $\frac{S}{r} \ll 1$,  we reach, after a straightforward computation,
\begin{equation}\label{Taylor-c1}
 c_1= \frac{1}{\sqrt S}\left(1+\frac{4}{15}\frac{S}{r}+\frac{2}{45}\frac{S^2}{r^2}+\cdots\right).
\end{equation}
In the sequel we denote
\begin{equation}\label{Taylor-c1-etoile}
 c_1^{*}:= \frac{1}{\sqrt S}\left(1+\frac{4}{15}\frac{S}{r}\right),
\end{equation}
the first order term of expansion \eqref{Taylor-c1}.

To verify the accuracy of our previsions, we ran simulations of the full system \eqref{syst-p-q-D},
the one established before simplification thanks to the quasi-linkage equilibrium approximation. We numerically estimate the instantaneous speed by following the movement of the center of the fronts. 
The comparison with the theoretical speed $\ep c_1^{*}=sc_1^{*}$ (let us recall that $s$ appearing in the original model is nothing else than $\varepsilon$, see \eqref{sAsBep}) is shown in Figure \ref{fig:speed-theory-numerics}.
We observe that formula \eqref{Taylor-c1-etoile} gives a very good approximation of the instantaneous speed when $\frac{S}{r}$ is not too small (right of the figure).
This validates {\it a posteriori} the quasi-linkage equilibrium approximation.

\begin{figure}[H]
\begin{center}
  \includegraphics[width=7cm]{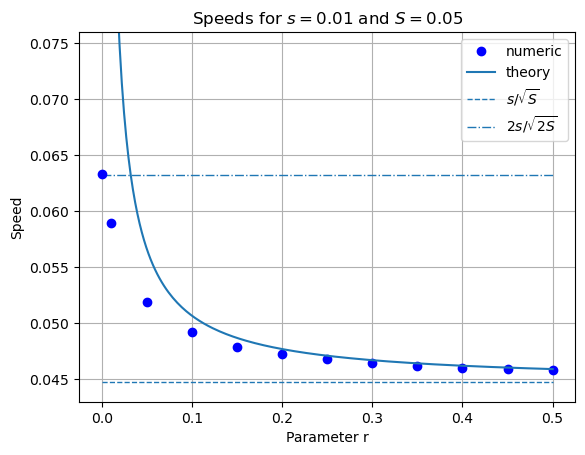}
\caption{
Comparison of the theoretical speed $sc_1^{*} $  and of the numerically estimated speed of the stacked fronts.
}
\label{fig:speed-theory-numerics}
\end{center}
\end{figure}

When $r=0.5$ (free recombination), the linkage disequilibrium stays small and the coupled clines move at a speed which is close to the one each cline would have if travelling alone, that is $s/\sqrt{S}$ (or $s\sigma/\sqrt{2S}$ in the original spatial scale, as obtained by Barton \cite{Barton79} in a single-locus model). Indeed, without  interaction, we are left with $u_t=u_{xx}+Sf(u)+sg(u)$, which is nothing else than the bistable equation ($0<\frac s S<1$)
$$
u_t=u_{xx}+Su(1-u)\left(2u-1+\frac sS\right),
$$
whose traveling wave, explicitly computed as $\frac 12-\frac 12 \tanh \left(\frac{\sqrt S}{2}(x-\frac{s}{\sqrt S}t)\right)$, has speed $s/\sqrt S$. 

At the other extreme, when $r=0$ (no recombination) the system becomes equivalent to a single locus where one allele has a fitness advantage $2s$ and with a cost for heterozygotes $2S$, leading to a bistable wave speed of $2s/\sqrt{2S}$.

When $r\in (0,0.5)$,  the speed of the coupled clines decreases monotonously as recombination $r$ increases. 

Our concluding remark is as follows:  whatever the values of the parameters, interacting and eventually stacked clines travel faster than one cline alone.

\section{Conclusion and perspectives}\label{s:conclusion}

In this paper we have investigated the solutions of equation \eqref{eq-u}, describing the dynamics of two coupled, asymmetric genetic incompatibilities (underdominant loci) with identical fitness effects, in a quasi-linkage equilibrium regime.
The two main results are as follows:
first, we have shown that when $\ep=0$, there is a unique standing wave $u_0$ under a normalization condition;
then, in Section \ref{s:front}, we have shown that when $\ep>0$ is small enough, there exists a traveling wave $u_{\ep}$ defined as a perturbation of $u_0$, and we obtained a simple approximation for its speed.

Those results were obtained under a series of assumptions that we recall here for discussion:
\begin{align} \label{all_assumptions}
s_A, s_B &< S \tag{H1} \\
s_A, s_B, S &\ll r \tag{H2} \\
S_A=S_B,\hspace{4pt}&\hspace{4pt} s_A = s_B \tag{H3} \\
p &= q.\tag{H4}
\end{align}

Assumption (H1) is the frame of this work which was devoted to the heterozygote inferior case. It is therefore not a hypothesis we want to discuss \emph{per se}.

Assumption (H2) expresses that we are in the case of small selective advantages.  When it does not hold, $D$ may not be small, in which case the quasi-linkage equilibrium approximation (that allowed us to reduce the number of variables) is no longer valid.
It can easily be seen that $-\frac{1}{4} \leq D \leq \frac{1}{4}$ always holds,
and that,  as shown by the $D$ equation in \eqref{syst-p-q-D}, positive $D$ is generated whenever $p$ and $q$ travel in the same direction (that is $p_xq_x>0$), while negative $D$ is generated otherwise.
These facts help to understand the kind of contribution $D$ makes to the coupling between $p$ and $q$ in \eqref{syst-p-q-D}.

Assumption (H3) is basically a hypothesis of exchangeability between loci. Although this allowed us to simplify the algebra, different loci should have different
fitness effects, and it would thus be of interest to relax this hypothesis.

Last but not least, assumption (H4) conveys the strong argument that the $A$ cline and the $B$ cline were stacked in the past and will remain stacked forever in the future.
This is indeed a good starting point from a mathematical perspective. Nevertheless, in the context of population genetics, more interesting questions arise when (H4) does not hold.
In such a situation, the coupling in \eqref{syst-p-q-D} can give rise to non-standard behaviours, such as {\it adaptation of the speed}.
The questions that arise are such as: can a traveling front be pinned by a standing front?
Will a front traveling at a large speed crossing a slower traveling front adapt its speed so as to remain stacked with the slower one?
A preliminary numerical exploration has shown that there can be a vast zoology of situations. We hope to present them in a future work.

\appendix

\section{Some useful results and tools}

We recall the {\it Implicit Function Theorem}, see \cite[Theorem 4.B]{Zei-86} for instance.

\begin{thm}[Implicit Function Theorem]\label{thm:implicit-functions}
	Let $X$, $Y$ and $Z$ be three Banach spaces. Suppose that:
	\begin{enumerate}[label={\rm (\roman*)}]
		\item The mapping $\mathcal F:U\subset X\times Y\to Z$ is defined on an open neighbourhood $U$ of $(x_0, y_0)\in X\times Y$ and $\mathcal F(x_0, y_0)=0$. 
		\item The partial Fr\'echet derivative of $\mathcal F$ with respect to $y$ exists on $U$ and
			$$
			\mathcal F_y(x_0, y_0):Y\to Z \text{  is bijective}.
			$$
		\item $\mathcal F$ and $\mathcal F_y$ are continuous at $(x_0,y_0)$. 
	\end{enumerate}
	Then, the following properties hold:
	\begin{enumerate}[label={\rm (\alph*)}]
		\item {\rm Existence and uniqueness}. There exist $r_0>0 $ and $r>0$ such that,  for every $x\in X$ satisfying $\Vert x-x_0\Vert\leq r_0$, there exists a unique $y(x)\in Y$ such that $\Vert y-y_0\Vert\leq r$ and $\mathcal F(x, y(x))=0$.
		\item {\rm Continuity}. If $\mathcal F$ is continuous in a neighbourhood of $(x_0,y_0)$, then the mapping  $x\mapsto y(x)$ is continuous in a neighbourhood of $x_0$. 
		\item {\rm Higher regularity}. If $\mathcal F$ is of the class $C^m$, $1\leq m\leq \infty$,  on a neighbourhood of $(x_0, y_0)$, then $x\mapsto y(x)$ is also of the class $C^m$ in a neighbourhood of $x_0$.
	\end{enumerate}
\end{thm}

In Section \ref{s:front} we apply Theorem \ref{thm:implicit-functions} to the operator $\mathcal F$ defined in \eqref{eq:F}, with $X=\R$, $x=\ep$, $x_0=0$, $Y=\R\times C^{2,\alpha}_\mu(\R)$, $y=(c,h)$, $y_0=(0,0)$, and $Z=C^{0,\alpha}_\mu(\R)$.

\bigskip

Next, we quote some results on Fredholm operators. Let us recall that the operator $L$ has the Fredholm property with index 0 if $\ker L$ has a finite dimension, $\rg L$ is closed and has finite codimension and
$$
\textrm{ ind } L:=\dim\ker L-\codim\rg L=0.
$$
In particular, since its range is closed, such an operator is {\it normally solvable}:
\begin{equation*}
	\exists u\neq 0, Lu=f \quad \Leftrightarrow\quad \forall \phi\in (\rg L)^\perp, \phi(f)=0,
\end{equation*}
and remark that $(\rg L)^\perp = \ker L^*$. 

We  recall below a theorem from Volpert, Volpert and Collet \cite[Theorem 2.1 and Remark p787]{Vol-Vol-Col-99}.

\begin{thm}[Fredholm property on the line]\label{thm:Fredholm-line} 
	For $0<\alpha<1$, consider the operator  $L:C^{2,\alpha}(\mathbb R)\to C^\alpha (\mathbb R)$  defined by
	\begin{equation*}
		Lu:=a(x)u''+b(x)u'+c(x)u,
	\end{equation*}
where the coefficients $a(x)$, $b(x)$, $c(x)$ are smooth, and $a(x)\geq a_0$ for some $a_0>0$. Assume further that the coefficients $a(x)$, $b(x)$, and $c(x)$ have finite limits as $x\to\pm\infty$ and denote
	\begin{equation*}
		a^\pm:=\lim_{x\to\pm\infty}a(x), \qquad b^\pm:=\lim_{x\to\pm\infty}b(x), \qquad c^\pm:=\lim_{x\to\pm\infty}c(x).
	\end{equation*}
	Finally, let us define the limiting operators
	\begin{equation*}
		L^\pm u:=a^\pm u''+b^\pm u'+c^\pm u,
	\end{equation*}
	and assume that for any $\lambda\geq 0$, the equation
	$$
	L^\pm u-\lambda u=0
	$$
	has no nontrivial solution in $ C^{2,\alpha}(\mathbb R)$.

	Then $L$ is Fredholm with index 0.
\end{thm}

Let us also recall a Fredholm property result for second-order ordinary differential equations, see the monograph of Volpert \cite[Chapter 9, Theorem 2.4 p. 366]{Vol-11}.

\begin{thm}[Fredholm property for second-order ODEs]\label{thm:Fredholm-ODE}
	With the notations of Theorem \ref{thm:Fredholm-line}, the operator  $L$ is Fredholm provided the two equations
	\begin{equation}\label{eq:test-Fredholm}
		-a^\pm\xi^2 + b^\pm i\xi + c^\pm=0
	\end{equation}
	has no real solution $\xi\in\mathbb R$. In this case the index of $L$ is given by the formula
	\begin{equation*}
		\text{\rm ind } L=\kappa^+-\kappa^-,
	\end{equation*}
	where $\kappa^\pm$ is the number of complex solutions to the characteristic equation 
	\begin{equation}\label{eq:char}
		a^\pm X^2-b^\pm X+c^\pm =0
	\end{equation}
	which have a positive real part.
\end{thm}

\begin{remark}[Fredholm property in weighted Hölder spaces]\label{rem:weighted-holder}
	We cannot directly apply Theorem \ref{thm:Fredholm-line} and Theorem \ref{thm:Fredholm-ODE} to our situation since we  consider the operator $L$ 
acting from $C^{2, \alpha}_\mu(\R)$ into $C^\alpha_\mu(\R)$, and not from $C^{2,\alpha}(\R)$ into $C^{0,\alpha}(\R)$. To circumvent this, we consider the operator $L^\mu:C^{2, \alpha}(\mathbb R)\to C^\alpha(\mathbb R)$ defined by:
	\begin{align} 
		L^\mu(u)&:=e^{\mu\sqrt{1+x^2}}L\left(ue^{-\mu\sqrt{1+x^2}}\right) \nonumber \\
		&=a(x)u'' + \left[ \frac{-2\mu x}{\sqrt{1+x^2}}a(x)+b(x)\right]u'\nonumber  \\
		&\quad+ \left[\left(\frac{\mu^2 x^2}{1+x^2}+ \frac{\mu x^2}{(1+x^2)^{\frac 32}}-\frac{\mu}{\sqrt{1+x^2}} \right)a(x)- \frac{\mu x}{\sqrt{1+x^{2}}}b(x)+c(x)\right]u.\label{Lmu}
	\end{align}
 Since  $T_\mu:u\in C^{2, \alpha}_\mu(\mathbb R) \mapsto e^{\mu\sqrt{1+x^2}}u\in  C^{2, \alpha}(\mathbb R)$ is continuously invertible, and $T_\mu^{-1}:u\in C^{0, \alpha}(\mathbb R) \mapsto e^{-\mu\sqrt{1+x^2}}u\in  C^{0, \alpha}_\mu (\mathbb R)$ is continuously invertible, the map $L=T_\mu^{-1}L^\mu T_\mu$ shares the same Fredholm property and index as $L^\mu$. 	As a result, if $L^\mu$ satisfies the assumptions of Theorem \ref{thm:Fredholm-line}, or Theorem \ref{thm:Fredholm-ODE}, then $L$ is a Fredholm operator with the same index as that of $L^\mu$.
\end{remark}

%\medskip

 %\noindent{\bf Acknowledgements.} %The authors are grateful to the anonymous referees whose precise comments have improved the presentation of the results. 

\bibliographystyle{plain}
\bibliography{biblio}

\end{document}